\documentclass{article}
\usepackage{amsmath, amssymb, latexsym}
\usepackage{amsopn,amsfonts, amsthm}
\usepackage[T1]{fontenc}

\newcommand{\zero}{\varepsilon^{\bullet}}
\DeclareMathOperator{\sgn}{sgn}

\DeclareMathOperator{\limes}{lim}
\newtheorem{theorem}{Theorem}

\newtheorem{corollary}[theorem]{Corollary}

\newtheorem{proposition}[theorem]{Proposition}

\theoremstyle{definition}

\newtheorem{definition}[theorem]{Definition}
\newtheorem{example}[theorem]{Example}

\newtheorem{problem}[theorem]{Problem}
\newtheorem{remark}[theorem]{Remark}

\title{The minimizing vector theorem\\ in symmetrized max-plus algebra}
\author{Cenap \"Ozel\and Artur Pi\k{e}kosz \and Eliza Wajch\and Hanifa Zekraoui\\
Department of Mathematics, King Abdulaziz University, \\
Jeddah, Kingdom of Saudi Arabia\\
Institute of Mathematics, Cracow University of Technology,\\
Cracow, Poland\\ 
Department of Mathematics and Physics,\\ University of Natural Sciences and Humanities in Siedlce,\\  
Siedlce, Poland \\
Department of Mathematics, Larbi Ben M'hidi University,\\
Oum El Bouaghi, Algeria\\}

\begin{document}

\maketitle

\begin{abstract} 
Assuming \textbf{ZF} and its consistency, we study some topological and geometrical properties of the
symmetrized max-plus algebra in the absence of the axiom of choice in order to discuss the minimizing vector
theorem for finite products of copies of the symmetrized max-plus algebra.  Several relevant statements that follow from the axiom of countable choice restricted to sequences of subsets of the real line are shown. Among them, it is proved that if all simultaneously complete and connected subspaces of the plane are closed, then the real line is sequential. A brief discussion about semidenrites is included. Older known proofs in \textbf{ZFC} of several basic facts relevant to proximinal and Chebyshev sets in metric spaces are replaced by new proofs in \textbf{ZF}. It is proved that a nonempty subset $C$ of the symmetrized max-plus algebra is Chebyshev in this algebra if and only if $C$ is simultaneously closed and connected. An application of it to a version of the minimizing vector theorem for finite products of the symmetrized max-plus algebra is shown. Open problems concerning some statements independent of \textbf{ZF} and other statements relevant to Chebyshev sets are posed.


\end{abstract}

\section{Introduction}

The Max-Plus Algebra (shortly: MPA), known also as the max-plus dioid (\cite{GM}) or  the tropical semiring (\cite{S}),  is the idempotent semiring (see Definition 1 in \cite{MP}) over the union of the set of all real numbers and one-point set consisting of $\varepsilon
=-\infty $, equipped with the maximum (as the addition) and the addition 
(as the multiplication). The zero and the unit of this semiring are respectively 
$\varepsilon $ and $0$. 

Despite the remarkable analogy between these two basic operations in MPA and 
the conventional algebra, the principal difference
reflects in  missing of inverse elements related to the first operation:
$\max$. A partial remedy is the Symmetrized Max-Plus Algebra (shortly: SMPA) which is an extension of MPA built by a method similar to that used in extending a given set of all natural numbers satisfying Peano's postulates to a set that can be called the set of all integers. 

Compared to
conventional linear algebra and linear system theory, the linear max-plus algebra is far
from fully developed, even though many papers and textbooks were
written on the topics of matrix calculations, among them, relevant to different types of
decomposition, spectral theory, linear systems and linear
optimization. Some important papers dealing with these subjects also were
written in SMPA; for instance, let us refer to \cite{SM}. 

In this paper, we give some topological and geometrical properties of SMPA in order to start a deeper investigation of the minimizing vector theorem in products of SMPA and point out new open problems relevant to Chebyshev sets. 

Although most mathematicians seem to assume \textbf{ZFC} as the axiomatic foundation, we realize that the axiom of choice (\textbf{AC}) is not needed to obtain the main results of our work. Therefore, the basic set-theoretic assumption of this study is \textbf{ZF}, formulated in \cite{Ku}, without any form of the axiom of choice. If it occurs necessary, we also use \textbf{AC} or another axiom which is a weaker form of \textbf{AC}; however, we clearly inform which statements in our article are independent of \textbf{ZF}.  We use the notation of \cite{Her} for \textbf{ZF} and axioms relevant to the axiom of choice. All topological notions not defined here can be found in \cite{ES}, \cite{Her} and \cite{M}. The book \cite{HoR} is the most useful source of information about validity and falsity in models of \textbf{ZF}. 

Our paper is organized as follows. Section 2 concerns basic properties of $d$-proximinal and $d$-Chebyshev sets in an arbitrary metric space $(X,d)$. New proofs in \textbf{ZF} to several facts on $d$-proximinality are shown. The proofs are more subtle than the already known ones in \textbf{ZFC}.   
We give a brief introduction about
calculation in MPA and the construction of max-plus algebra of pairs (MPAP)
 in subsection 3.1. In subsection 3.2, we give a short presentation of the construction of SMPA from MPAP and some remarks about computation in SMPA. In section 4, we define equivalent Euclidean and inner metrics in the set $\mathbb{S}$ of all points of SMPA. Equipped with the natural topology induced by the inner metric, the space $\mathbb{S}$ is a semidendrite which is embeddable in the complex plane $\mathbb{C}$. Some properties of semidendrites in \textbf{ZF} are shown. Section 5 concerns finite products $\mathbb{S}^n$ of $\mathbb{S}$. Natural metrics in $\mathbb{S}^n$ are considered. A notion of a geometric segment in $\mathbb{S}^n$ is introduced and it is shown that geometric segments can be identified with suitable broken lines in $\mathbb{C}^n$. Section 6 is about completeness and Cantor-completeness of the metrics in $\mathbb{S}^n$ defined in section 5. A few set-theoretic open problems are posed. New equivalents of the independent of \textbf{ZF} statement that $\mathbb{R}$ is sequential are found. In section 7, we introduce the structure of a semimodule in $\mathbb{S}^n$. Section 8 is devoted to distinct concepts of convex sets in $\mathbb{S}^n$. The final section 9 is about the minimizing vector theorem in $\mathbb{S}^n$. We give here a complete, satisfactory characterization of Chebyshev sets in $\mathbb{S}$. We also investigate finite products of Chebyshev sets in $\mathbb{S}$. We pose open problems on the minimizing vector theorem. Our paper is a preparation to deeper investigations of Chebyshev sets in $\mathbb{S}^n$ in the future.
 
 We shall use the set $\omega$ of all finite ordinal numbers of von Neumann as the set of all non-negative integers (see \cite{Ku}). 

\section{Chebyshev sets in metric spaces-preliminaries}
The aim of this section is to establish terminology and basic properties of Chebyshev sets in an arbitrary metric space. An extensive bibliography on Chebyshev sets in normed spaces which includes also survey articles is given in \cite{FM}, so it is not necessary to repeat this bibliography here.

\subsection{$d$-proximinal and $d$-Chebyshev sets}

Throughout this section, we assume that $(X, d)$ is a metric space. As usual, for $x\in X$ and a positive real number $r$, we denote $B_d(x, r)=\{ y\in X: d(x, y)<r\}$ and $\bar{B}_d(x, r)=\{ y\in X: d(x, y)\leq r\}$. The symbol $\tau(d)$ stands for the topology induced by $d$.  For simplicity,  we denote by $X$ the topological space $(X, \tau(d))$. If $Y$ is a subset of $X$, we shall consider $Y$ as a metric subspace of $(X, d)$ and as a topological subspace of $(X, \tau(d))$. The following definition can be regarded as an adaptation of Definition 2.3 of \cite{FM} to arbitrary metric spaces.  If it is not stated otherwise, $\mathbb{R}^n$ will stand for the $n$-dimensional Euclidean space equipped with the Euclidean metric induced by the standard inner product in $\mathbb{R}^n$.

\begin{definition} Let $K\subseteq X$.
\begin{enumerate}
\item[(i)] We define a set-valued mapping $P_{d,K}: X\to \mathcal{P}(X)$ by $P_{d, K}(x)=\{ y\in K: d(x, y)=d(x, K)\}$ if $K\neq\emptyset$ and by $P_{d, K}(x)=\emptyset$ if $K=\emptyset$, where $x\in X$.
\item[(ii)] We say that a $d$-\textbf{nearest point} to $x\in X$ in $K$ or, equivalently, \textbf{ the best $d$-approximation} of $x$ in $K$  is a point which belongs to the set $P_{d,K}(x)$.
\item[(iii)] The set $K$ is called \textbf{$d$-proximinal} in $X$ if $P_{d, K}(x)\neq\emptyset$ for all $x\in X$.
\item[(iv)] The set $K$ is called a \textbf{$d$-Chebyshev set} in $X$ if $P_{d,K}(x)$ is a singleton for each $x\in X$.
\item[(v)] If $x\in X$ is such that $P_{d,K}(x)$ is a singleton, we denote by $p_{d,K}(x)$ the unique point of $P_{d,K}(x)$ and we call $p_{d,K}(x)$ the \textbf{$d$-projection} of the point $x$ onto $K$. 
\item[(vi)] If $K$ is a $d$-Chebyshev set in $X$, then the mapping $p_{d, K}:X\to K$ is called the \textbf{$d$-projection of $X$ onto $K$}. If this does not lead to misunderstanding, $d$-projections will be called \textbf{metric projections}. 
\end{enumerate}
\end{definition} 

\begin{definition} Let $Y$ be a subset of the metric space $(X, d)$ and let $K\subseteq Y$. We say that $K$ is \textbf{$d$-proximinal in $Y$} if, for each $y\in Y$, the set $P_{d,K}(y)$ is nonempty. If, for each $y\in Y$, the set $P_{d,K}(y)$ is a singleton, we say that $K$ is \textbf{$d$-Chebyshev in $Y$}.
\end{definition}

A set which is $d$-Chebyshev in a subspace $Y$ of the metric space $(X, d)$ need not be $d$-Chebyshev in $X$. (Take a non-Chebyshev subset $Y$ in $\mathbb{R}^2$. Then 
$Y$ is Chebyshev in itself.)

Although it has been written a lot about Chebyshev sets in normed vector spaces, in general, it is not known much about the nature of $d$-Chebyshev sets for an arbitrary metric $d$; however, the concepts of proximinal and Chebyshev sets in metric spaces, especially in strictly convex metric spaces,  were investigated in the past, for instance, in \cite{ANT}, \cite{N} and \cite{NS}.  Unfortunately, even in the case when $d$ is induced by the inner product of a Hilbert space over the field $\mathbb{R}$,  versions of the axiom of choice have been involved in proofs of some properties of $d$-Chebyshev sets, while it is more desirable to investigate whether a statement is provable in $\mathbf{ZF}$ or it can fail in a model of $\mathbf{ZF}$. For instance, the statement that $\mathbb{R}^n$ is a Fr\'echet space was used in the proof on page 363 of \cite{Web} that every Chebyshev set in the $n$-dimensional Euclidean space $\mathbb{R}^n$ is closed. However, in view of Theorem 4.54 of \cite{Her}, that $\mathbb{R}$ is Fr\'echet is an equivalent of $\mathbf{CC}(\mathbb{R})$ which asserts that every nonempty countable collection of nonempty subsets of $\mathbb{R}$ has a choice function (see Definition 2.9 (1) of $\mathbf{CC}(\mathbb{R})$ in \cite{Her} and Form 94 in \cite{HoR}). In Cohen's original model of $\mathbf{ZF}$ (see model $\mathcal{M}1$ in \cite{HoR}), $\mathbb{R}$ is not Fr\'echet (see Form 73 of \cite{HoR}). Therefore, the proof in \cite{Web} that Chebyshev sets in $\mathbb{R}^n$ are closed is not a proof in $\mathbf{ZF}$ but it is a too complicated proof in $\mathbf{ZF}+\mathbf{CC}(\mathbb{R})$. If we replace the norm by the metric $d$ in Proposition 2.9 of \cite{FM} and in its proof in \cite{FM}, then we can see that the following more general fact than Proposition 2.9 of \cite{FM} has a trivially simple proof in $\mathbf{ZF}$ which is included below to show its simplicity:

\begin{proposition} Each $d$-proximinal set in a nonempty metric space $(X,d)$ is nonempty and closed in $(X, d)$. In particular, each $d$-Chebyshev set in the metric space $(X, d)$ is closed in $(X,d)$.
\end{proposition}
\begin{proof} Let $A$ be a $d$-proximinal set in $X$. It is obvious that $A\neq\emptyset$ because $X\neq\emptyset$. Let $x\in\text{cl}_d(A)$ and $a\in P_{d,A}(x)$. Then $d(x, a)=d(x, A)=0$, so $x\in A$ because $x=a$. Therefore, $A=\text{cl}_d(A)$. 
\end{proof}

\begin{definition} The  \textbf{$d$-graph} of a set $A\subseteq X$ is the set 
$$\text{Gr}_d(A)=\bigcup_{x\in X}[\{x\}\times P_{d, A}(x)].$$
\end{definition}

\begin{proposition} The $d$-graph of a closed subset of $X$ is closed in $X\times X$.
\end{proposition} 
\begin{proof} Let $K$ be a closed subset of $X$. If $K=\emptyset$, then $\text{Gr}_d(K)=\emptyset$. Assume that $K\neq\emptyset$. Let $(x,y)\in (X\times X)\setminus \text{Gr}_d(K)$. Then $y\notin P_{d,K}(x)$. If $y\notin K$, then $U=X\times (X\setminus K)$ is a neighbourhood of $(x,y)$ such that $U\cap\text{Gr}_d(K)=\emptyset$. Now, suppose  that $y\in K$. Then $d(x,y)>d(x, K)$ because $y\notin P_{d,K}(x)$. Let $r$ be a positive real number such that $3r<d(x,y)-d(x,K)$. Put $V=B_d(x, r)\times B_d(y, r)$. Suppose that $(z,t)\in V\cap\text{Gr}_d(K)$. Then $t\in P_{d,K}(z)$, so $d(z,t)=d(z, K)$. Moreover, $d(x,y)\leq d(x,z)+d(z,t)+d(t,y)<d(z,K)+2r\leq d(z,x)+d(x,K)+2r\leq d(x,K)+3r$.This gives that $d(x,y)-d(x, K)<3r$ which is impossible. Hence $V\cap\text{Gr}_d(K)=\emptyset$. In consequence, $(X\times X)\setminus\text{Gr}_d(K)$ is open in $X\times X$.
\end{proof}

Propositions 3 and 5 taken together imply the following theorem:

\begin{theorem} If $K$ is a $d$-proximinal subset of $X$ then the $d$-graph of $K$ is closed in $X\times X$. In particular, if $K$ is a $d$-Chebyshev set in $X$, then the $d$-projection of $X$ onto $K$ has a closed graph.
\end{theorem}

Theorem 6 is an extension and generalization of Corollary 2.19 from \cite{FM} and, for instance, of Theorem 4 from \cite{N}. However, the proofs to Corollary 2.19 in \cite{FM} and to Theorem 4 in \cite{N} are not proofs in $\mathbf{ZF}$ because, in $\mathbf{ZF}$,  only sequential closedness of Chebyshev sets was shown in \cite{FM} and \cite{N}, while it is known that a sequentially closed subset of a metric space may fail to be closed in a model of $\mathbf{ZF}$. Namely, as in Theorem 4.55 of \cite{Her}, we denote by 
$\mathbf{CC}(c\mathbb{R})$ the following statement: Every sequence of nonempty complete metric subspaces of $\mathbb{R}$ has a choice function. 
In view of Theorem 4.55 of \cite{Her}, that $\mathbb{R}$ is sequential is an equivalent of $\mathbf{CC}(c\mathbb{R})$. 

\begin{remark} In Cohen's original model (see model $\mathcal{M}1$ in \cite{HoR}), for each natural number $n\ge 1$, the space $\mathbb{R}^n$ is not sequential because it contains non-separable subspaces (see Form 92 of \cite{HoR}). 
\end{remark}
 
It is known that $\mathbf{CC}(\mathbb{R})$ is strictly stronger than $\mathbf{CC}(c\mathbb{R})$ in $\mathbf{ZF}$ (see pages 74-77 of \cite{Her}).

\begin{definition} Let $K$ be a nonempty subset of $X$. We say that the mapping $P_{d,K}: X\to P(X)$ is \textbf{$d$-continuous at $x\in X$} if $P_{d, K}(x)$ is  a singleton and the following condition is satisfied:
$$\forall_{\varepsilon\in(0;+\infty)}\exists_{\delta\in (0;+\infty)}\forall_{y, z \in X}[(d(x,y)<\delta \wedge z\in P_{d,K}(y))\Rightarrow d(z, p_{d,K}(x))<\varepsilon].$$

\end{definition} 

Definition 8 is a modification of Definition 2.23 of \cite{FM}. In the case when $K$ is a nonempty $d$-Chebyshev set in $X$, then the metric projection $p_{d,K}$ of $X$ onto $K$ is continuous if and only if $P_{d,K}$ is $d$-continuous  at each $x\in X$ in the sense of our Definition 8, while $p_{d,K}$ is sequentially continuous if and only if $P_{d,K}$ is continuous  at each $x\in X$ in the sense of Definition 2.23 of \cite{FM} under the additional assumption that $X$ is a normed vector space and $d$ is induced by the norm of $X$.  In general, given metric spaces $(X, d)$ and $(Y, \rho)$, the sequential continuity of a mapping $f: X\to Y$ is not  equivalent to the continuity of $f$ in $\mathbf{ZF}$ (see Theorem 4.54 of \cite{Her} and Form 94 in \cite{HoR} which fails in model $\mathcal{M}1$ of \cite{HoR}). It can de deduced from the proof in $\mathbf{ZF}$ to Theorem 3.15 in \cite{Her} that the following theorem holds true in $\mathbf{ZF}$:

\begin{theorem} If $X$ is a separable metrizable space and $Y$ is a metrizable space, then a mapping $f: X\to Y$ is continuous if and only if $f$ is sequentially continuous.
\end{theorem}

\begin{remark} Let $K$ be a $d$-Chebyshev set in a normed vector space $X$ where $d$ is induced by the norm of $X$. If $X$ is finitely dimensional, the classical proofs of the continuity of the metric projection $p_{d,K}$ of $X$ onto $K$ given, for instance, in \cite{FM} and \cite{Web} have a gap in $\mathbf{ZF}$ because only sequential continuity of $p_{d, K}$ was shown there. To deduce from the sequential continuity of $p_{d,K}$ its continuity, we need Theorem 9. If $X$ is infinitely dimensional, $p_{d,K}$ may fail to be continuous (see Example 2.54 in \cite{FM} and  page  246 in \cite{G}).
\end{remark}

We are going to give a direct proof in $\mathbf{ZF}$ of the continuity of $p_{d,K}$ if $K$ is a $d$-Chebyshev set in an arbitrary nonempty metric space $(X, d)$ having the Heine-Borel property.

\begin{definition} Let $Y$ be a subset of the metric space $(X,d)$. We say that:
\begin{enumerate}
\item[(i)] $Y$ has the \textbf{Heine-Borel property with respect to $d$} or, equivalently, $Y$ is \textbf{$d$-totally complete} if, for each $K\subseteq Y$, it is true that $K$ is compact if $K$ is simultaneously closed in $Y$ and $d$-bounded;
\item[(ii)] $Y$ is \textbf{$d$-boundedly compact} if, for each positive real number $r$ and for each $x\in X$, the set $Y\cap\bar{B}_d(x, r)$ is compact.
\end{enumerate}
\end{definition}

The Heine-Borel property, called also total compactness by some authors (see \cite{N}),  is a familiar notion in topology and in the theory of metric spaces. Definition 11(ii) is an adaptation of Definition 2.25 of \cite{FM} to metric spaces,  given here in order to notice that a subset $Y$ of the metric space $(X, d)$ is $d$-boundedly compact if and only if $Y$  has the Heine-Borel property with respect to $d$.  

\begin{theorem} Let $K$ be a $d$-boundedly compact subset of the metric space $(X,d)$ and let $x\in X$ be such that $P_{d,K}(x)$ is a singleton. Then the mapping $P_{d, K}$ is $d$-continuous at $x$.
\end{theorem}
\begin{proof} Suppose that $P_{d,K}$ is not $d$-continuous at $x$. There exists a positive real number $\varepsilon$ such that, for each $n\in\omega \setminus \{ 0\}$, the set 
$$F_n=[\bigcup\{P_{d,K}(y): y\in B_d(x, \frac{1}{n})\}]\setminus B_d(p_{d,K}(x),\varepsilon)$$
 is nonempty. Let $n\in\omega \setminus \{ 0\}$ and $y\in B_d(x,\frac{1}{n})$. For any $z\in P_{d,K}(y)$, we have: $d(x, z)\leq d(x, y)+d(y,z)\leq \frac{1}{n}+ d(y,K)\leq\frac{1}{n}+d(y,x)+d(x, K)\leq\frac{2}{n}+d(x, K)$. This implies that the set $F_n$ is $d$-bounded. Hence the set $C_n=\text{cl}_{X}(F_n)$ is $d$-bounded. Then $C_n\cap K$ is compact because $K$ is $d$-boundedly compact. Since $F_n\subseteq C_n\cap K$, we obtain that $C_n\subseteq K$. Of course, the collection $\{C_n: n\in\omega \setminus \{ 0\}\}$ is centred. Thus, it follows from the compactness of $C_n$ that there exists $z_0\in\bigcap_{n\in\omega \setminus \{ 0\}}C_n$. Then $z_0\in K$. Let $g(t)=d(t, K)$ for $t\in X$. Let us fix a positive real number $\eta$. Since the function $g$ is continuous, there exists $n_0\in\omega \setminus \{ 0\}$ such that $\frac{1}{n_0}<\varepsilon$ and if $t\in B_d(x, \frac{1}{n_0})$, then $\vert g(t)-g(x)\vert <\eta$. Since $z_0\in\text{cl}_X(F_n)$, there exists $z_1\in F_n$ such that $d(z_0, z_1)<\frac{1}{n}$. There exists $y_1\in B_d(x, \frac{1}{n})$ such that $z_1\in P_{d,K}(y_1)$. Since $z_0\in K$, we have that $d(x, K)\leq d(x, z_0)$. On the other hand, $d(x, z_0)\leq d(x, y_1)+ d(y_1, z_1)+d(z_1, z_0)\leq\frac{2}{n}+d(y_1,K)\leq\frac{2}{n}+d(x,K)+\eta$. This, together with the arbitrariness of  the numbers $\eta>0$ and $n\ge n_0$, implies that $d(x, z_0)\leq  d(x, K)$. In consequence, we obtain that $d(x, z_0)=d(x, K)=d(x, p_{d,K}(x))$, so $z_0=p_{d,K}(x)$. However, $\varepsilon\leq d(p_{d, K}(x), z_1)\leq d(p_{d,K}(x), z_0)+d(z_0, z_1)\leq d(p_{d, K}(x), z_0)+\frac{1}{n}$ and thus $d(p_{d,K}(x), z_0)>0$. The contradiction obtained completes the proof.
\end{proof}

An application of our proof to Theorem 12 is that the following immediate consequence of Theorem 12 holds true in $\mathbf{ZF}$: 
 
\begin{corollary} If $K$ is a $d$-boundedly compact $d$-Chebyshev set in $X$, then the $d$-projection of $X$ onto $K$ is continuous.
\end{corollary}

\subsection{Operations on $d$-Chebyshev sets}

As in subsection 2.1, let us assume that $(X, d)$ is a nonempty metric space. 

\begin{proposition} Let $\mathcal{C}$ be a nonempty collection of $d$-Chebyshev sets in $X$ and let $K=\bigcup\mathcal{C}$. Suppose that $K$ is  $d$-boundedly compact and has the property that, for each pair $x,y$ of points of $K$, there exists $C\in\mathcal{C}$ such that $x,y\in C$. Then $K$ is a $d$-Chebyshev set in $X$.
\end{proposition}
\begin{proof} Let $x_0\in X$. Take any $c\in K$. Put $r=d(x_0, c)$ and $E=\bar{B}_d(x_0, r)\cap K$. Then $d(x_0, K)=d(x_0, E)$. Since $E$ is nonempty and compact, there exists $c_0\in E$ such that $d(x_0, c_0)=d(x_0, E)$, so $K$ is $d$-proximinal. Suppose that $c_1, c_2\in K$ and $d(x_0, c_1)=d(x_0, c_2)=d(x_0, K)$. There exists $C\in\mathcal{C}$ such that $c_1, c_2\in C$. Let $i\in\{1, 2\}$ and let $a=p_{d, C}(x_0)$. Then $d(x_0, c_i)\leq d(x_0, a)$ because $a\in K$. Since $c_i\in C$, we have $d(x_0, a)\leq d(x_0, c_i)$. Hence $d(x_0, c_i)=d(x_0, a)$. This implies that $c_i=a$ becauce $P_{d, C}(x_0)=\{a\}$. 
\end{proof}

The following notions have appeared useful for investigations of Chebyshev sets, especially in strictly convex metric spaces (see, for instance, \cite{ANT}, \cite{N} and \cite{NS}):

\begin{definition} Let $Y\subseteq X$ and let $x, y\in Y$. Then:
\begin{enumerate}
\item[(i)] (cf. Definition 2.1 of \cite{NS})  the \textbf{$d$-segment} in $Y$ between points $x$ and $y$ is the set $[x, y]_{d, Y}$ defined by:
$$[x, y]_{d,Y}=\{ z\in Y: d(x, z)+d(z, y)=d(x, y)\};$$
\item[(ii)] if $Y=X$, then $[x, y]_d=[x, y]_{d, Y}$;
\item[(iii)] (cf. Definition 2.6 of \cite{NS}) a set $A\subseteq Y$ is called \textbf{$d$-convex} in $Y$ if, for each pair of points $x, y\in A$, the inclusion $[x, y]_{d,Y}\subseteq A$ holds;
\item[(iv)] (cf. Definition 2.1 of \cite{NS})  a point $z\in X$  is \textbf{$d$-between $x$ and $y$} in $Y$ if $z\in [x, y]_{d,Y}$.
%\item[(v)] (cf. page 72 of \cite{N}) a \textbf{$d$-midpoint of $x$ and $y$} is a point $z\in X$ such that $d(x, z)=d(z, y)=\frac{1}{2}d(x, y)$.
\end{enumerate}
\end{definition}

The following proposition, relevant to Corollary 1 of \cite{ANT},  follows immediately from Proposition 14:

\begin{proposition} If all $d$-segments in $X$ are $d$-Chebyshev sets, then every  nonempty $d$-boundedly compact $d$-convex set in $X$ is $d$-Chebyshev in $X$.
\end{proposition} 

\begin{proposition} Suppose that $A_1, A_2$ and $A=A_1\cup A_2$ are all $d$-Chebyshev sets in $X$. Then, for each $x\in X$, we have:
\begin{eqnarray}
p_{d,A}(x)=\left\{ 
\begin{array}{ccc}
p_{d,A_1}(x) \text{ if } d(x, p_{d, A_1}(x))\leq d(x, p_{d, A_2}(x)) \\ 
p_{d,A_2}(x) \text{ \ \ \ \ \ \ \ \ \ \ \ \ \ otherwise}
\end{array}
\right. \text{.}
\end{eqnarray}
\end{proposition}
\begin{proof} Let $x\in X$. Denote $a=p_{d,A}(x)$ and $a_i=p_{d,A_i}(x)$ for $i\in\{1, 2\}$. Of course, $d(x, a)\leq d(x, a_i)$ for $i\in\{1, 2\}$. Suppose that $a\in A_1$. Then $d(x, a_1)\leq d(x, a)$,  so $d(x, a)=d(x, a_1)$. This implies that $a=a_1$ and $d(x, a_1)\leq d(x, a_2)$. If $a\in A_2\setminus A_1$, then $a=a_2$ and $d(x, a_1)>d(x, a)$. 
\end{proof}

%%To shorten our notation, we shall use the set $\omega$ of all finite ordinal numbers of von Neumann as the set of all non-negative integers (see \cite{Ku}). 
%%%%%Let $\omega \setminus \{ 0\}=\omega\setminus\{0\}$ where  $0=\emptyset$.

Using similar elementary arguments as in the proof to Proposition 17,  one can show that the following proposition holds:

\begin{proposition} Let $n\in\omega\setminus\{0\}$. Assume that, for each $i\in n$, we are given a $d$-Chebyshev set  $A_i$ in $X$ such that the set $A=\bigcup_{i\in n}A_i$ is $d$-Chebyshev in $X$. Let $x\in X$ and let $a=p_{d,A}(x)$. Then $a=p_{d,A_i}(x)$ for each $i\in n$ such that $a\in A_i$.
\end{proposition} 

Now, we assume that $n\in\omega \setminus \{ 0\}$ and $(X_i, d_i)$ is a nonempty metric space for each $i\in n$, while $X=\prod_{i\in n}X_i$. Usually, the equivalent metrics $\rho_k$ for $k\in 3$ are considered in $X$ where, for each $x, y\in X$, we have:
\begin{eqnarray*}
\rho_0(x, y)& = & \max\{d_i(x(i), y(i)): i\in n\},\\
\rho_1(x, y)&=&\sqrt{\sum_{i\in n}d_i(x(i), y(i))^2},\\
\rho_2(x, y)& = & \sum_{i\in n}d_i(x(i), y(i)).
\end{eqnarray*}

Since the theorem below will be of essential importance, we include its simple proof for completeness:

\begin{theorem} Let $A_i\subseteq X_i$ for each $i\in n$ and let $A=\prod_{i\in n}A_i$, while $k\in\{1, 2\}$. Then the following conditions are satisfied:
\begin{enumerate}
\item[(i)] For each $x\in X$, the equality $P_{\rho_k, A}(x)=\prod_{i\in n}P_{d_i, A_i}(x(i))$ holds.
\item[(ii)] $A$ is a $\rho_k$-Chebyshev set in $X$ if and only if $A_i$ is a $d_i$-Chebyshev set in $X_i$ for each $i\in n$.
\end{enumerate}
\end{theorem}
\begin{proof} To prove $(i)$, assume first that $a\in P_{\rho_k, A}(x)$. Let $j\in n$. Then $a(j)\in A_j$. Consider any $y_j\in A_j$ and define $y\in X$ as follows: $y(i)=a(i)$ for each $i\in n\setminus\{ j\}$, while $y(j)=y_j$. Since $y\in A$, we have $\rho_k(x, a)\leq \rho_k(x, y)$. This implies that $d_j(x(j), a(j))\leq d_j(x(j), y_j)$, so $a(j)\in P_{d_j, A_j}(x(j))$ and, in consequence,  $a\in \prod_{i\in n}P_{d_i, A_i}(x(i))$. 

Now, assume that $b\in \prod_{i\in n}P_{d_i, A_i}(x(i))$. Then $b\in A$. Moreover, if $z\in A$, then, since $d_i(x(i), b(i))\leq d_i(x(i), z(i))$ for each $i\in n$, we have that $\rho_k(x, b)\leq \rho_k(x, z)$. Hence $b\in P_{\rho_k, A}(x)$. This completes the proof to $(i)$. 

That $(ii)$ also holds is an immediate consequence of $(i)$. 
\end{proof}

The following example shows that Theorem 19 cannot be extended to the metric $\rho_0$.
\begin{example} Let $ A_0=A_1=[-1,1]$. For $x, y\in\mathbb{R}$, let $d_0(x,y)=d_1(x,y)=\vert x-y\vert$. Then, for $x, y\in\mathbb{R}^2$, we have $\rho_0(x,y)=\max\{\vert x(0)-y(0)\vert, \vert x(1)-y(1)\vert\}$. Now, let $x^{\ast}=(0,2)$.  Then $\{0\}=P_{d_0,A_0}(0)$ and $\{1\}=P_{d_1, A_1}(2)$, while, for $A=A_0\times A_1$, the set $P_{\rho_0,A}(x^{\ast})$ is infinite because $P_{\rho_0,A}(x^{\ast})=[-1,1]\times \{1\}$.
\end{example}
 
\section{Calculations in dioids of interest}

In this section, we recall the construction and calculations in algebraic structures of our interest. Readers can compare this material with \cite{SM}.

\subsection{Calculations in MPA and MPAP}

The max as the addition operation and the addition as the multiplication
operation are denoted respectively by $\oplus $ and $\otimes $. The union 
$\mathbb{R} \cup \left\{ \varepsilon \right\} $ (where $\varepsilon=-\infty\notin \mathbb{R}$ is a fixed element)  is denoted by 
$\mathbb{R}_{\varepsilon }$\ and the semiring 
$\left( \mathbb{R}_{\varepsilon },\oplus ,\otimes \right) $ is denoted by 
$\mathbb{R}_{\max }$. So, for any $a$, $b\in \mathbb{R}_{\varepsilon }$, 
\begin{eqnarray}
\text{ }a\oplus b=\max \left( a,b\right) \text{ and }\ a\otimes b=a+b\text{.}
%\tag{1}
\end{eqnarray}

The structure $\mathbb{R}_{\max}$ is a dioid (=an idempotent semiring) and a semifield (see Definition 1 in \cite{MP}). Notice that
$\mathbb{R}^n_{\max}$ is naturally a semimodule over $\mathbb{R}_{\max}$.

Analogously to conventional algebra, scalars in linear combinations are written before vectors. The rules for the order of evaluation of the tropical operations are similar to those of conventional algebra, i.e. max-plus algebraic power has the highest
priority, and the tropical multiplication has a higher priority than
the tropical addition. All rules of calculation in 
$\mathbb{R}_{\max }$ can be illustrated in the following example:

\begin{example}
\begin{eqnarray*}
2\oplus \left( 3^{5}\oplus 2^{-1}\right) \otimes 1\oplus \varepsilon ^{2}
&=&2\oplus \left( 5\times 3\oplus \left( -2\right) \right) \otimes 1\oplus
\varepsilon \\
&=&2\oplus \left( 15\oplus \left( -2\right) \right) \otimes 1 \\
&=&2\oplus \left( 15\otimes 1\oplus \left( -2\right) \otimes 1\right) \\
&=&2\oplus \left( \left( 15+1\right) \oplus \left( -2+1\right) \right) \\
&=&2\oplus \left( 16\oplus \left( -1\right) \right) \\
&=&\max \left( 2,16,-1\right) =16
\end{eqnarray*}
\end{example}

There is no inverse related to max operation, which makes a simple equation 
$a\oplus x=b$ often insolvable. For this reason, we first consider the algebra of
pairs as a tool in the construction of the Symmetrized Max-Plus Algebra. 

Let 
$\mathcal{P}_{\varepsilon }=$ $
\mathbb{R}_{\varepsilon }\times 
\mathbb{R}_{\varepsilon }$ with the operations $\oplus $ and $\otimes $ defined by
operations in (1) as follows: 
\begin{eqnarray}
\left( a,b\right) \oplus \left( c,d\right) &=&\left( a\oplus c,b\oplus
d\right) %\tag{2}
 \\
\left( a,b\right) \otimes \left( c,d\right) &=&\left( a\otimes c\oplus
b\otimes d,a\otimes d\oplus b\otimes c\right)
%\tag{3}
\end{eqnarray}

It is easy to verify that $\left( \mathcal{P}_{\varepsilon },\oplus ,\otimes
\right) $ is a dioid with $\left( \varepsilon ,\varepsilon
\right) $ and $\left( 0,\varepsilon \right) $ as its zero and unit,
respectively. Notice that $\left( \varepsilon ,\varepsilon \right) $ is the absorbing
element for $\otimes $. This structure is called the \textbf{Max-Plus Algebra of
Pairs} (abbreviated to MPAP) and denoted by $\mathcal{P}_{\max }$. As 
$\mathbb{R}_{\max }$ is isomorphic to the subdioid of pairs $\left( a,\varepsilon
\right) $, $a\in \mathbb{R}_{\varepsilon }$, then\ $\mathbb{R}_{\max }$ itself can be considered as a subdioid of $\mathcal{P}_{\max }$. If $u=(x,y)\in \mathcal{P}_{\varepsilon }$, then one defines the\textbf{ max-plus norm} (or \textbf{absolute
value}) of $u$ as $|u|_{\oplus }=x\oplus y$ and we have two unary
operators: $\ominus $ (the max-plus algebraic minus operator) and 
$($\textperiodcentered $)^{\bullet }$ (the balance operator) such that 
$\ominus  u=(y,x)$ and $u^{\bullet }=u\oplus (\ominus u)$ $=(|u|_{\oplus} ,
|u|_{\oplus})$. The properties of these operators can be expressed in the following
proposition:

\begin{proposition}[cf. \cite{SM}]
For all $u$ and $v$ in $\mathcal{P}_{\varepsilon }$, we have:
\begin{enumerate}
\item[(i)] $u^{\bullet }=\left( \ominus u\right) ^{\bullet }=\left( u^{\bullet
}\right) ^{\bullet }$;
\item[(ii)] $u\otimes v^{\bullet }=\left( u\otimes v\right) ^{\bullet }=u^{\bullet}
\otimes v=u^{\bullet}\otimes v^{\bullet}$; 
\item[(iii)] $\ominus \left( \ominus u\right) =u$, $\ominus \left( u\oplus v\right)
=(\ominus u)\oplus (\ominus v)$ and\\ 
$(\ominus u)\otimes v=\ominus \left(u\otimes v\right) =u\otimes (\ominus v)$.
\end{enumerate}
\end{proposition}
 
 \begin{remark}
 The above can be interpreted: $($\textperiodcentered $)^{\bullet }$ is  a
 projection onto another subdioid of $\mathcal{P}_{\varepsilon }$ consisting of fixed points of $\ominus$ 
  and
  $\ominus $ is an involution commuting with $\oplus $.
 \end{remark}

We shall write $u\ominus v$ instead
of $u\oplus (\ominus v)$.

\subsection{Construction of SMPA}

For readers' convenience, let us recall in brief the construction of SMPA described in \cite{SM}. 

Since $u\ominus u=(|u|_{\oplus },|u|_{\oplus })=\left( \varepsilon
,\varepsilon \right) \Rightarrow u=\left( \varepsilon ,\varepsilon \right)$, 
then $\mathcal{P}_{\max }$ is, as 
$
\mathbb{R}_{\max }$, without inverses with respect to $\oplus $. For this reason, the
balance relation $\bigtriangledown $ is introduced: For all $u=\left(
a,b\right)\in\mathcal{P}_{\varepsilon} $ and $v=\left( c,d\right) \in \mathcal{P}_{\varepsilon }$, we have 
\begin{eqnarray}
u\bigtriangledown v\Leftrightarrow a\oplus d=b\oplus c\text{.}
\end{eqnarray}
This relation is reflexive and symmetric, but not transitive because, for real numbers $x, y, z$ such that
$x>z>y$, we have that 
\[
x=x\oplus x=x\oplus y\text{ }\neq y\oplus z\text{, }
\]
which gives $\left( x,y\right) \bigtriangledown \left( x,x\right) $ and 
$\left( x,x\right) \bigtriangledown \left( z,x\right) $, but $\left(
x,y\right) $ is not balanced with $\left( z,x\right) $. So we need to modify
the relation $\bigtriangledown$ somewhat to have an associated with it equivalence relation $\sim $ in $\mathcal{P}_{\varepsilon}$ such that it is possible, in a standard way, to extend the operations $\oplus $ and $\otimes $ of $\mathcal{P}_{\max }$
to the quotient set $\mathbb{S}=\mathcal{P}_{\varepsilon }/_{\sim }$ to obtain the semiring $\mathcal{P_{\max}}/_{\sim}$  which
contains not inverses but symmetrized elements which have some properties of
inverses.  Let the equivalence relation $\sim $ be defined as follows: for all $u,v\in\mathcal{P}_{\varepsilon}$ with
$u=\left( a,b\right) $ and $v=\left( c,d\right)$ where $a,b,c,d\in\mathbb{R}_{\varepsilon}$, we have
\begin{eqnarray}
u\sim v\Leftrightarrow \left\{ 
\begin{array}{ccc}
u & \bigtriangledown  & v\text{ if }a\neq b\text{ and }c\neq d \\ 
u & = & v\text{ \ \ \ \ \ \ \ \ \ \ \ \ \ otherwise}
\end{array}
\right. \text{.}
\end{eqnarray}
The structure $\left( \mathbb{S},\oplus ,\otimes \right)=\mathcal{P}_{\max}/_{\sim} $ is
called the \textbf{Symmetrized Max-Plus Algebra} (shortly: SMPA) and denoted by 
$\mathbb{S}_{\max }$. The semifield  $\mathbb{R}_{\max }$ can be identified with a 
subdioid of $\mathbb{S}_{\max }$, and $\mathbb{S}_{\max }$ can
be considered as a semimodule over 
$\mathbb{R}_{\max }$.

We have three types of equivalence classes of $\sim $  which are called max-plus positive, max-plus negative and balanced elements, respectively. Namely, for $u\in\mathcal{P}_{\varepsilon}$, let $\overline{u}=\{v\in\mathcal{P}_{\varepsilon}: u\sim v\}$. Then, for $a\in\mathbb{R}$, we get  
\begin{eqnarray}
\overline{\left( a,\varepsilon \right) } &=&\left\{ \left( a,b\right) \in 
\mathcal{P}_{\max }\text{: }a>b\right\}\text{,} \\
\overline{\left( \varepsilon ,a\right) } &=&\left\{ \left( a,b\right) \in 
\mathcal{P}_{\max }\text{: }a<b\right\}\text{.}
\end{eqnarray}
Moreover, for each $a\in\mathbb{R}_{\varepsilon}$, we get
\begin{eqnarray}
\overline{\left( a,a\right) } &=&\left\{ \left( a,a\right)\right\}\text{.}
\end{eqnarray}

If $a\in\mathbb{R}$, the element  $\oplus a=\overline{\left( a, \varepsilon\right)}$ of $\mathbb{S}$ is called \textbf{ max-plus positive}, while the element $\ominus a=\overline{\left(\varepsilon, a\right)}$ of $\mathbb{S}$ is called \textbf{max-plus negative}. If $a\in\mathbb{R}_{\varepsilon}$, the element $a^{\bullet }=\overline{\left( a, a\right)}$ of $\mathbb{S}$ is called \textbf{balanced}. The element $\varepsilon^{\bullet}$ is the \textbf{zero element} of $\mathbb{S}_{\max}$. Let 
$\mathbb{S}^{\oplus }$ be
the set of all elements of $\mathbb{S}$ which are either max-plus positive or equal to $\varepsilon^{\bullet}$. Let 
$\mathbb{S}^{\ominus }$ be the set of  all elements of $\mathbb{S}$ which are either max-plus negative or equal to $\varepsilon^{\bullet}$. Finally, let $\mathbb{S}^{\bullet }$ be the set of all balanced elements of $\mathbb{S}$. 
The set $\mathbb{S}$ can be written as the union $\mathbb{S}^{\oplus }\cup \mathbb{S}^{\ominus }\cup \mathbb{S}^{\bullet }$. The sets $\mathbb{S}^{\oplus}, \mathbb{S}^{\ominus}$ and $\mathbb{S}^{\bullet}$ have only one 
common point $\varepsilon^{\bullet}$. Let us put $\oplus\varepsilon=\ominus\varepsilon=\varepsilon^{\bullet}$.  We can identify $\mathbb{S}^{\oplus }$ with 
$\mathbb{R}_{\max}$ by assigning $\oplus a$ to $a\in\mathbb{R}_{\varepsilon}$.  Similarly, $\mathbb{S}^{\ominus }$ can be identified with 
$\ominus \mathbb{R}_{\max}=\{\ominus a: a\in\mathbb{R}_{\varepsilon}\}$, and $\mathbb{S}^{\bullet }$ can be identified with 
$(\mathbb{R}_{\max})^{\bullet}=\{a^{\bullet}: a\in\mathbb{R}_{\varepsilon}\}$.

\begin{definition}[cf. \cite{SM}]
Let $a\in \mathbb{S}$. The \textbf{max-plus positive part} $a^{\oplus }$ and the
\textbf{max-plus negative part} $a^{\ominus }$ of $a$ are defined as follows:
\begin{enumerate}
\item[(i)] if $a\in $ $\mathbb{S}^{\oplus }$, then, $a^{\oplus }$ $=a$ and $a^{\ominus
}=\varepsilon $; 
\item[(ii)] if $a\in $ $\mathbb{S}^{\ominus }$ then $a^{\oplus }$ $=$ 
$\varepsilon $ and $a^{\ominus }=$ $\ominus a$;
\item[(iii)] if $a\in $ $\mathbb{S}^{\bullet }$, then there exists $x\in 
\mathbb{R}_{\varepsilon }$ such that $a=$ $\overline{(x,x)}$ and then $a^{\oplus }=$ 
$a^{\ominus }$ $=x$.
\end{enumerate}
\end{definition}

It is clear from the previous definition that, for each $a\in \mathbb{S}$,  the following equality holds: 
 \begin{eqnarray}
 a=a^{\oplus }\ominus a^{\ominus }.
 \end{eqnarray}
This decomposition is unique if we allow only $a=x\ominus y$  
$(x,y\in \mathbb{S}^{\oplus})$  with either 
$x=\varepsilon$ or $y=\varepsilon$ or $x=y$.

\section{Topology and geometry in SMPA}

\subsection{Metrics on SMPA}

In this paragraph, we introduce  two distances on~$\mathbb{S}$.

\begin{proposition}[cf. \cite{SM}]
The formula
 \begin{eqnarray}
\left\vert a\right\vert=a^{\oplus}\oplus a^{\ominus }
\end{eqnarray}
defines an absolute value on $\mathbb{S}_{\max}$ as a dioid 
(or a norm with $\mathbb{S}_{\max}$ understood as a semimodule over 
the semifield $\mathbb{R}_{\max}$).
\end{proposition}

\begin{proof}
Step 1). 
If $\left\vert a\right\vert=\varepsilon $, then 
\[
\max \left( a^{\oplus },a^{\ominus }\right) =-\infty \Leftrightarrow
a^{\oplus }=a^{\ominus }=-\infty \Leftrightarrow a=\varepsilon \text{.}
\].

Step 2). Naturally, we have
\[
\left\vert a\otimes b\right\vert=\left( a\otimes b\right)
^{\oplus }\oplus \left( a\otimes b\right) ^{\ominus }\text{.}
\]
We discuss three cases:

If $a$, $b\in \mathbb{S}^{\oplus }$, then 
\[
a=a^{\oplus }\text{, }b=b^{\oplus }\Rightarrow \left( a\otimes b\right)
^{\oplus }=a^{\oplus }\otimes b^{\oplus }\text{, }\left( a\otimes b\right)
^{\ominus }=\varepsilon 
\]
and 
\[
\left\vert a\right\vert=a^{\oplus }\text{, }\left\vert
b\right\vert=b^{\oplus }\text{,}
\]
which gives 
\[
\left\vert a\otimes b\right\vert=a^{\oplus }\otimes b^{\oplus
}=\left\vert a\right\vert\otimes \left\vert b\right\vert \text{.}
\]

If $a$, $b\in \mathbb{S}^{\ominus }$, then $a^{\ominus }=\ominus a$ and 
$a^{\oplus }=\varepsilon $ , $b^{\ominus }=\ominus b$ and $b^{\oplus}=
\varepsilon $, which yields 
\[
\left( \ominus a\right) \otimes \left( \ominus b\right) =\left( a\otimes
b\right) ^{\oplus }\text{,}
\]
then, 
\[
\left( a\otimes b\right) ^{\ominus }=\varepsilon \text{, }\left\vert
a\right\vert =\ominus a\text{, }\left\vert b\right\vert =\ominus b\text{,}
\]
which gives 
\[
\left\vert a\otimes b\right\vert =\left( a\otimes b\right)
^{\oplus }=\left( \ominus a\right) \otimes \left( \ominus b\right)
=\left\vert a\right\vert \otimes \left\vert b\right\vert \text{.}
\]

If $a\in \mathbb{S}^{\oplus }$, $b\in \mathbb{S}^{\ominus }$, then 
$a=a^{\oplus }$, $b^{\ominus }=\ominus b$ and $b^{\oplus }=\varepsilon $,
which yields 
\[
a\otimes \left( \ominus b\right) =\ominus \left( a\otimes b\right) =\left(
a\otimes b\right) ^{\ominus }\text{,}
\]
then, 
\[
\left( a\otimes b\right) ^{\oplus }=\varepsilon \text{, }\left\vert
a\right\vert =a,\left\vert b\right\vert =\ominus b\text{,}
\]
which gives
\[
\left\vert a\otimes b\right\vert =\left( a\otimes b\right)
^{\ominus }=a\otimes \left( \ominus b\right) =\left\vert a\right\vert \otimes 
\left\vert b\right\vert \text{.}
\]
So in all cases, the $\left\vert \cdot \right\vert $ is multiplicative.

Step 3). Similarly, we can prove the triangle inequality (equality, in fact).
\end{proof}

\begin{remark}
The max-plus absolute value $\left\vert \cdot \right\vert $ does not
produce a distance on $\mathbb{S}$ in the similar way as the usual absolute value 
produces a distance on $\mathbb{R}$ 
 because, if we define $\rho(x, y)=\vert x\ominus y\vert$ for $x, y\in\mathbb{S}$, then,  for $a\in\mathbb{S}$ with $a\neq \zero $, we get $
\rho\left( a,a\right) =\left\vert a\ominus a\right\vert =|a^{\bullet}|=|a|\neq \varepsilon$.
\end{remark}

\begin{corollary}
Each $a\in \mathbb{S}$ can be written as
\begin{eqnarray}
 a = \sgn(a) \: |a|
 \end{eqnarray}
where $\sgn(a)\in \{ \oplus, \ominus, ^{\bullet} \}$. If $a\neq \varepsilon^{\bullet}$, then this formula is  unique. 
\end{corollary}

\begin{corollary}
One can write the formulas for actions in $\mathbb{S}$:
\begin{eqnarray}
a\oplus b= \sgn(a\oplus b) ( \: |a|\oplus |b|),\\
a\otimes b= \sgn(a)\: \sgn(b)\: (|a|\otimes |b|)
\end{eqnarray}
where $\sgn(a \oplus b)=\left\{
\begin{array}{ll}
\sgn(a), & |a|>|b|,\\ 
\sgn(b), & |b|>|a|,\\
\sgn(a), & |a|=|b|, \sgn(a)= \sgn(b),\\
\bullet, & |a|=|b|, \sgn(a)\neq \sgn(b),\end{array}
\right.$\\
and the multiplication table for signs is obvious.
\end{corollary}

\begin{definition} The \textbf{canonical embedding} of $\mathbb{S}$ into $\mathbb{C}$ is the mapping  $\Phi :\mathbb{S} \to \mathbb{C}$ defined in the following way:
for  $e^{-\infty}=0$, $\theta=\frac{-1+\sqrt{3}i}{2}$ and any $r\in\mathbb{R}_{\varepsilon}$, we put
$$\Phi(\oplus r)=\theta e^r, \quad \Phi(\ominus r)=\theta^2 e^r, \quad \Phi(r^{\bullet})= e^r.$$
\end{definition}

We denote the Euclidean metric in $\mathbb{R}^n$ by $d_{e,n}$ or by $d_{e}$ when it clear what $n$ is. In particular, for the plane $\mathbb{C}=\mathbb{R}^2$, we have $d_e=d_{e,2}$.  As we have already mentioned, $\mathbb{R}^n$ will be usually considered with its Euclidean metric $d_e$ and with the natural topology induced by $d_e$. 

\begin{definition} The \textbf{natural} (or \textbf{usual}) topology of $\mathbb{S}$ is the smallest topology in $\mathbb{S}$ which makes the mapping $\Phi:\mathbb{S}\to\mathbb{C}$ continuous.
\end{definition}

 We shall consider two natural metrics in $\mathbb{S}$ that are defined below:

\begin{definition}
Let $d_1:$ $\mathbb{S\times S}\mathbb{\longrightarrow }\mathbb{R}$ be defined as follows:  
$$d_1(a,b)=d_e(\Phi(a),\Phi(b))$$
for all $a, b\in\mathbb{S}$.  The metric $d_1$ in $\mathbb{S}$  will be called the \textbf{Euclidean distance} in $\mathbb{S}$. 
\end{definition}

\begin{definition}
Let $d_2: \mathbb{S\times S}\mathbb{\longrightarrow }\mathbb{R}$ be defined as follows: 
\[
d_2\left( a,b\right) =\left\{ 
\begin{array}{ccc}
\left\vert \exp |a|-\exp |b|\right\vert & \text{if} & \sgn(a)=\sgn(b) \\ 
\exp |a|+\exp |b| & \text{if} & \sgn(a)\neq \sgn(b)
\end{array}
\right. \text{} 
\]
for all $a, b\in\mathbb{S}$.  The metric $d_2$  in $\mathbb{S}$ will be called the \textbf{inner distance} in $\mathbb{S}$.
\end{definition}

\begin{remark}
It is clear that both Euclidean and inner distances in $\mathbb{S}$ are equivalent and they induce the natural topology of $\mathbb{S}$; moreover, the metrics $d_1$ and $d_2$ are uniformly equivalent and equal  on each of the lines $\mathbb{S}^{\oplus}$, $\mathbb{S}^{\ominus}$, $\mathbb{S}^{\bullet}$. As a topological space, we shall consider $\mathbb{S}$ only with its natural topology. The space $\mathbb{S}$ is homeomorphic with the subspace $\Phi(\mathbb{S})$ of $\mathbb{C}=\mathbb{R}^2$. The mapping $\Phi$ is an isometry of the metric space $(\mathbb{S}, d_1)$ onto the metric subspace $\Phi(\mathbb{S})$ of $(\mathbb{C}, d_e)$. The metric space $(\mathbb{S}, d_2)$ is isometric with:

1)  the metric subspace $(\mathbb{R}\times\{0\})\cup(\{0\}\times [0, +\infty))$ of the metric space $(\mathbb{R}^2, \rho_2)$ where $\rho_2(x, y)=\vert x(0)-y(0)\vert+\vert x(1)-y(1)\vert$ for all $x, y\in\mathbb{R}^2$;

2) the metric subspace $\{z\in \mathbb{R}^2:  z(1)=\vert z(0)\vert \vee (z(0)=0 \wedge z(1) <0 \}$ of the metric space $(\mathbb{R}^2, \rho_0)$ where $\rho_0(x, y)=
\max(\vert x(0)-y(0)\vert,\vert x(1)-y(1)\vert )$ for all $x, y\in\mathbb{R}^2$.
\end{remark}

Let $X$ be a topological space. We recall that it is said that a point $c\in X$ \textbf{disconnects} $X$ \textbf{between points} $a, b\in X$ if there exists a pair $U, V$ of disjoint open sets in $X$ such that $X\setminus\{ c\}= U\cup V$ and $a\in U$, while $b\in V$. The space $X$ is a \textbf{dendrite} if it is a continuum such that, for each pair $a, b$ of distinct points of $X$, there exists a point $c\in X$ which disconnects $X$ between points $a$ and $b$.  It is said that $X$ is a \textbf{semicontinuum} if, for each pair $a, b$ of points of $X$, there is a continuum $C$ in $X$ such that $a, b\in C$.  Let us say that  $X$  is a \textbf{semidendrite} if it is a semicontinuum such that, for each pair $a, b$ of distinct points of $X$, there exists a point $c\in X$ which disconnects $X$ between $a$ and $b$.  

\begin{proposition} The space $\mathbb{S}$ is a semidendrite.
\end{proposition}
\begin{proof} It suffices to notice that $\mathbb{S}$ is homeomorphic with the semidendrite $\Phi(\mathbb{S})$ on the plane $\mathbb{C}$.
\end{proof}

Basic properties of dendrites are nicely described in \cite{Ch} and \cite{M}. Some of the properties can be easily proved in \textbf{ZF} and they are also properties of semidendrites. For example, that every semidendrite is a Hausdorff space and that every semicontinuum contained in a semidendrite is a semidendrite are statements obviously true in \textbf{ZF}.  However, the axiom of choice is heavily used, for instance,  in an adaptation to  metrizable semidendrites of the known proof for metrizable dendrites of the following theorem:

\begin{theorem} It holds true in \textbf{ZFC} that, for each pair $a, b$ of distinct points of a metrizable semidendrite $X$,  there exists in $X$ exactly one arc with end-points $a$ and $b$. 
\end{theorem}

\begin{problem} Can $\mathbf{ZFC}$ be replaced by $\mathbf{ZF}$ in Theorem 35?
\end{problem}

\subsection{Continuity of the inner operations}

In this paragraph, we study the continuity and the discontinuity of the inner
operations in $\mathbb{S}_{\max}$ with respect to the usual topology in ~$\mathbb{S}$. 
The following proposition is obvious:

\begin{proposition}
The operation $\otimes $ is continuous in the usual topology of~$\mathbb{S}$.
\end{proposition}

\begin{proposition}
The operation $\oplus $ is discontinuous in the usual topology of 
$\mathbb{S}$.
\end{proposition}

\begin{proof}
For $r\in \mathbb{R}$, we have discontinuity of $\oplus : \mathbb{S}\times \mathbb{S}\to \mathbb{S}$ at $(\oplus r,\ominus r)$,  $(\ominus r, \oplus r)$, $(\oplus r, r^{\bullet})$,  $( r^{\bullet}, \oplus r)$, $(\ominus r, r^{\bullet})$,  and $( r^{\bullet}, \ominus r)$. For example:
$$ \limes\limits_{s\to r^-} \oplus(\oplus s, r^{\bullet})=r^{\bullet}, \qquad
  \limes\limits_{s\to r^+} \oplus(\oplus s, r^{\bullet})=\oplus r.
 $$
\end{proof} 

\begin{remark}
Let $\tau$ be the quotient topology induced by the mapping $|\cdot |: \mathbb{S}\to \mathbb{R}_{\varepsilon}$ where $\mathbb{R}_{\varepsilon}$ is equipped with its standard topology. The topology $\tau$ is not $T_0$; however, both operations $\oplus$ and $\otimes$ from $(\mathbb{S}, \tau)\times(\mathbb{S}, \tau)$ to $(\mathbb{S}, \tau)$ are continuous.
\end{remark}

\section{Topology and geometry of finite products of SMPA}

Let $n\in\omega \setminus \{ 0\}$. Then the set $\mathbb{S}^n$ of all functions $x: n\to\mathbb{S}$  is identified with the $n$-Cartesian product of $\mathbb{S}$. In particular, $\mathbb{S}^1$ is identified with $\mathbb{S}$. The \textbf{usual} or \textbf{natural topology} of $\mathbb{S}^n$ is the product topology of the space $\mathbb{S}^n$ where $\mathbb{S}$ is equipped with its usual topology (see Definition 30). %Let us denote by 

\subsection{The main metrics in $\mathbb{S}^n$}

We consider the Euclidean metric $d_1$ and the inner metric $d_2$ in $\mathbb{S}$. For $j\in\{1, 2\}$,  it is most natural to associate with the distance function $d_j$ the following metrics $\rho_{0,j}, \rho_{1,j}$ and $\rho_{2,j}$ in $\mathbb{S}^n$:
\begin{eqnarray*}
\rho_{0,j}(x, y)& = & \max\{d_j(x(i), y(i)): i\in n\},\\
\rho_{1,j}(x, y)&=&\sqrt{\sum_{i\in n}d_j(x(i), y(i))^2},\\
\rho_{2,j}(x, y)& = & \sum_{i\in n}d_j(x(i), y(i))
\end{eqnarray*}
 where $x,y\in\mathbb{S}^n$. Of course, the metrics $\rho_{k,j}$  are all equivalent and they induce the natural topology of $\mathbb{S}^n$. We shall pay a special attention to the metrics $\rho_{1,1}$ and $\rho_{1,2}$.
 
\begin{definition} The metric $\rho_{1,1}$ will be denoted by $D_1$ and called the \textbf{Euclidean metric} in $\mathbb{S}^n$, while the metric $\rho_{1,2}$ will be denoted by $D_2$ and called the \textbf{inner metric} in $\mathbb{S}^n$.
 \end{definition}
 
\begin{definition} For the canonical embedding $\Phi$ of $\mathbb{S}$ into $\mathbb{C}$, let $\Phi_i=\Phi$ where $i\in n$. Then the product $\prod_{i\in n}\Phi_i$ will be denoted by $[^n\Phi]$ and called the \textbf{canonical embedding} of $\mathbb{S}^n$ into $\mathbb{C}^n$.
\end{definition}

\begin{remark} The mapping $[^n\Phi]$ is an isometry of the metric space $(\mathbb{S}^n, D_1)$ onto the metric subspace $[^n\Phi](\mathbb{S}^n)$ of $(\mathbb{C}^n, d_{e,2n})$ where the set $\mathbb{C}^n$ is identified with $\mathbb{R}^{2n}$ in an obvious way.
\end{remark}
 
\begin{remark} For $x, y\in\mathbb{S}^n$, let $L_{x,y}$ be the shortest broken line in $\mathbb{C}^n$ which connects $[^n\Phi](x)$ with 
$[^n\Phi](y)$ and is contained in $[^n\Phi](\mathbb{S}^n)$. Then $D_2(x, y)$ is the length in $\mathbb{C}^n$ of the curve $L_{x,y}$, while
$$[x, y]_{D_2}=[^n\Phi]^{-1}(L_{x,y}).$$  
\end{remark}

\subsection{Geometric segments as broken lines}

For $x, y\in\mathbb{R}^n$ (or $x, y\in\mathbb{C}^n$, resp.), let $[x, y]=\{ (1-t)x+ty: 0\leq t\leq 1\}$ be the traditional segment with end-points $x, y$.  

\begin{definition} Let $x, y\in\mathbb{S}^n$. Then:
\begin{enumerate}
\item[(i)]  the $D_2$-segment in $\mathbb{S}^n$ between $x$ and $y$ will be denoted by $[x, y]_g$ and called the \textbf{geometric segment} between $x$ and $y$;
\item[(ii)] if the traditional segment $[[^n\Phi](x), [^n\Phi](y)]$ in $\mathbb{C}^n$ is a subset of $[^n\Phi](\mathbb{S}^n)$, then $[^n\Phi]^{-1}([[^n\Phi](x), [^n\Phi](y)])$ will be called the \textbf{traditional segment} in $\mathbb{S}^n$ between $x$ and $y$.
\end{enumerate}
\end{definition}

For an arbitrary geometric segment $[a, b]_g$, let us demonstrate a method of finding a possibly small finite collection $\mathcal{C}$ of traditional segments such that $[a, b]_g=\bigcup\mathcal{C}$.

\begin{definition} For distinct $u,v\in\{\oplus, \ominus, \bullet\}$, we define an embedding $\Psi_{(u,v)}:\mathbb{S}^{u}\cup\mathbb{S}^{v}\to\mathbb{R}$ as follows:  for each $r\in\mathbb{R}$, we put ${\bullet}r=r^{\bullet}$,  
$\Psi_{(u,v)}(ur)= e^{r}$, while $\Psi_{(u,v)}(vr)= -e^{r}$. We put $\Psi_{(u,v)}(\varepsilon^\bullet)=0$. Moreover, let $\Psi_{(u, u)}:\mathbb{S}^u\to [0; +\infty)$ be defined by the equality $\Psi_{(u,u)}(x)=\Psi_{(u, v)}(x)$ for each $x\in\mathbb{S}^u$ and each $v\in\{\oplus, \ominus, \bullet\}\setminus\{u\}$.
\end{definition}

Now, let us fix $n\in\omega \setminus \{ 0\}$ and let $a, b\in\mathbb{S}^n$ be given. Assume that $a\neq b$.  For each $j\in n$ we choose an ordered pair $(u(j),v(j))\in\{\oplus, \ominus, \bullet\}\times\{\oplus, \ominus, \bullet\}$ such that $a(j)\in S^{u(j)}$ and $b(j)\in S^{v(j)}$. Let $\Psi=\prod_{j\in n}\Psi_{(u(j), v(j))}$. To simplify notation, put $\Psi_j=\Psi_{(u(j),v(j))}$ for $j\in n$. Consider the traditional segment $[\Psi(a), \Psi(b)]$ in $\mathbb{R}^n$. Let  $J$ be the set of all $j\in n$ which have the property that there exists a unique $t_j\in [0, 1]$ such that $(1-t_j)\Psi_j(a(j))+t_j\Psi_j(b(j))=0$ and let $K\subseteq J$ be a maximal subset of $J$ such that, for all $l,p\in K$, $t_l\neq t_p$ if $l\neq p$.  Suppose that the set $K$ is nonempty and it has exactly $k$ elements. Let $j_0, \dots, j_{k-1}$ be the injective sequence of all elements of $K$ such that $t_{j_l}\leq t_{j_p}$ if $l, p\in k$ and $l\leq p$. Put $x_p= (1-t_{j_p})\Psi(a)+t_{j_p}\Psi(b)$ for each $p\in K$. We get the broken line $L=[\Psi(a), x_0]\cup\bigcup_{p\in k-1}[x_p, x_{p+1}]\cup [x_{k-1}, \Psi(b)]$. Then $\Psi^{-1}(L)$ is  the geometric segment $[a, b]_g$, while the length $\vert L\vert$ of $L$ is equal to $D_2(a, b)$.  Now, it is easily seen that every geometric segment in $\mathbb{S}^n$ is a union of at most $n+1$ traditional segments. The following example shows how to use this method in practice. 

\begin{example}
Let us consider points $a, b\in\mathbb{S}^3$ where $a=(\oplus 0, \ominus\text{ln}3, [\text{ln}2]^{\bullet})$ and $b=(\ominus 0,  0^{\bullet} ,\oplus 0)$. To follow the method described above, we put $u(0)=\oplus, v(0)=\ominus, u(1)=\ominus, v(1)=\bullet,  u(2)=\bullet, v(2)=\oplus$ and $\Psi=\prod_{j\in 3}\Psi_{(u(j), v(j))}$. For $t\in [0, 1]$, we have $ (1-t)\Psi(a)+t\Psi(b)=(1-2t, 3-4t, 2-3t)$. Then $K=\{0, 1, 2\}$, while $t_0=\frac{1}{2}, t_1=\frac{3}{4},  t_2=\frac{2}{3}$. In consequence, $x_0=(0, 1, \frac{1}{2}), x_1=(-\frac{1}{3}, \frac{1}{3}, 0), x_2=(-\frac{1}{2},0, -\frac{1}{4})$, For the broken line $L=[\Psi(a),x_0]\cup [x_0, x_1]\cup [x_1, x_2]\cup[x_2, \Psi(b)]$, we have $\vert L\vert=D_2( a, b)=\sqrt{29}$. In this case,  $[a, b]_g$ is a union of four traditional segments in $\mathbb{S}^3$.
\end{example}

\begin{remark} Let $u,v\in\{\oplus, \ominus, \bullet\}$ be distinct. Suppose that $a\in\mathbb{S}^u\setminus\{\varepsilon^\bullet\}$, while $b\in\mathbb{S}^v\setminus\{\varepsilon^\bullet\}$. Then $[a, b]_{d_1}=\{a, b\}$. 
\end{remark}

\section{Completeness in $\mathbb{S}^n$}
 
According to \cite{Gut} and \cite{Ker}, to consider completeness in metric spaces in the absence of the axiom of choice, it is necessary to make a distinction between the notions of metric completeness and Cantor completeness. 

\subsection{Completeness and Cantor completeness of metrics}

To make our terminology clear, we recall the following definitions with an additional notation which is more suitable to our aims:
\begin{definition}  Let $(X, d)$ be a metric space and let $Y\subseteq X$. It is said that:
\begin{enumerate}
\item[(i)]  the set $Y$  is \textbf{complete} in $(X, d)$ or, equivalently, \textbf{$d$-complete} in $X$ if each $d$-Cauchy sequence of points of $Y$ converges in $(X, d)$ to a point from $Y$;
\item[(ii)] the set $Y$ is \textbf{Cantor complete} in $(X,d)$ or, equivalently, $Y$ is \textbf{Cantor $d$-complete} in $X$ if, for each sequence $(F_n)$ of nonempty $d$-closed in $Y$ sets such that $F_{n+1}\subseteq F_n$ for each $n\in\omega \setminus \{ 0\}$ and $\lim_{n\to +\infty} \text{diam}_d(F_n)=0$, we have $\bigcap_{n\in\omega \setminus \{ 0\}} F_n\neq\emptyset$ where $\text{diam}_d(F_n)$ denotes the diameter of $F_n$ in the metric space $(X,d)$. 
\end{enumerate}
\end{definition}

All Cantor complete subspaces of the metric space $(X, d)$ are closed in $(X, d)$ but, in some models of $\mathbf{ZF}$,  complete subsets of metric spaces need not be closed (see \cite{Gut},  Proposition 6 of \cite{Ker} and Theorem 4.55 of \cite{Her}).  

Let us recall that $\mathbf{CC}$ is the axiom of countable choice which states that every nonempty countable collection of nonempty sets has a choice function (see Definition 2.5 of \cite{Her} and Form 8  in \cite{HoR}). In the light of Theorem 7 of \cite{Ker}, that every complete metric space is Cantor complete is an equivalent of $\mathbf{CC}$. 

\subsection{Completeness in $\mathbb{S}$}

\begin{theorem}
Let $i\in\{1,2\}$. Then $(\mathbb{S}, d_i)$ is a  Cantor complete metric space. Therefore, all  closed subspaces of $\mathbb{S}$ are $d_i$-complete. That all  $d_i$-complete subspaces of $\mathbb{S}$ are closed in $\mathbb{S}$ is an equivalent of $\mathbf{CC}(c\mathbb{R})$. 
\end{theorem}
\begin{proof} The subspace $\Phi(\mathbb{S})$ of the metric space $(\mathbb{C}, d_e)$ is Cantor complete. This, together with Remark 33, implies that the metric spaces $(\mathbb{S}, d_1)$ and $(\mathbb{S}, d_2)$ are Cantor complete.  Now, suppose that $\mathbf{CC}(c\mathbb{R})$ fails. Then, by Theorem 4.55 of \cite{Her},  $\mathbb{R}$ is not sequential. Let $X$ be a subset of ${\mathbb{R}}_+$ such that $X$ is sequentially closed but not closed in $\mathbb{R}$. Let us consider $X$ as a metric subspace of the complete metric space $(\mathbb{R}, d)$ where $d(x, y)=\vert x-y\vert$ for $x, y\in\mathbb{R}$. Then the metric $d$ is complete on $X$. Since the metric space $X$ is isometric with a metric subspace of the metric subspace $\Phi(\mathbb{S}^\oplus)$ of the semidendrite $\Phi(\mathbb{S})$, it follows that, for $i\in\{1, 2\}$,  not all complete subspaces of $(\mathbb{S}, d_i)$ are closed in $\mathbb{S}$. 

Finally, let us assume that $\mathbf{CC}(c\mathbb{R})$  holds. Then, by  Theorem 4.55 of \cite{Her}, all complete subspaces of $\mathbb{R}$ are closed in $\mathbb{R}$. This implies that all complete subspaces of $(\mathbb{S}, d_i)$ are closed in $\mathbb{S}$ where $i\in\{1, 2\}$.
 \end{proof}
 
 \begin{remark}
 For $i\in\{1, 2\}$,  it is independent of $\mathbf{ZF}$ that all $d_i$-complete subspaces of 
 $\mathbb{S}$ are closed in $\mathbb{S}$. 
 In the light of Remark 7, for instance, in model $\mathcal{M}1$ of \cite{HoR}, there exist metric subspaces of $(\mathbb{S}, d_i)$ that are simultaneously complete and not closed. 
\end{remark}

 \begin{remark} There is a mistake in condition 7 of Theorem 4.54 of \cite{Her}. Namely, condition 7 of Theorem 4.54 on page 74 of \cite{Her} is the following statement:\\
\centerline{ (*)\emph{ Each second countable topological space is separable.}}\\
It is worth to notice that (*) is an equivalent of $\mathbf{CC}$.
To show that (*) implies $\mathbf{CC}$, one can use the method of the proof to Theorem 4.64 in \cite{Her}. It is well known by topologists and set theorists that $\mathbf{CC}$ implies (*). 
Condition 7 of Theorem 4.54 of \cite{Her} should be replaced by the following statement: \\
\centerline{\emph{Every second countable $T_0$-space is separable.}}  \\
We know from a private communication of S. da Silva with E. Wajch that Horst Herrlich was aware of his mistake after \cite{Her} had appeared in print.
 \end{remark}
 
\subsection{Completeness of the main metrics of $\mathbb{S}^n$}

One can deduce from Theorem 49 and from the definitions of $\rho_{k, i}$ (see subsection 5.1) that the following theorem holds: 

\begin{theorem}
Let $i\in\{1, 2\}$ and $k\in 3$. Then, for each $n\in\omega \setminus \{ 0\}$, the metric space $(\mathbb{S}^n, \rho_{k,i})$ is Cantor complete, so all closed subspaces of $\mathbb{S}^n$ are $\rho_{k,i}$-complete.  
\end{theorem}

\begin{problem} Let $i\in\{1, 2\}, k\in 3$. Consider any $n\in\omega \setminus \{ 0\}$ with $n>1$.  Does $\mathbf{CC}(c\mathbb{R})$ imply in $\mathbf{ZF}$ that all $\rho_{k,i}$-complete subspaces of $\mathbb{S}^n$ are closed in $\mathbb{S}^n$?
\end{problem}

A satisfactory solution to Problem 53 is unknown to us and it may be complicated. However, the following proposition follows easily from known facts:

\begin{proposition} Let $i\in\{ 1, 2\}$, $k\in 3$ and let $n\in\omega \setminus \{ 0\}$. If $\mathbf{CC}(\mathbb{R})$ holds, then
each $\rho_{k,i}$- complete subspace of $\mathbb{S}^n$ is closed in $\mathbb{S}^n$.
\end{proposition}
\begin{proof} Suppose that $\mathbf{CC}(\mathbb{R}) $ holds. Let $X\subseteq\mathbb{S}^n$. In view of Theorem 4.54 of \cite{Her}, the subspace $X$ of $\mathbb{S}^n$ is separable. If $X$ is not closed in $\mathbb{S}^n$, it follows from the separability of $X$ and from the first-countability of $\mathbb{S}^n$ that there exists a sequence of points of $X$ which converges in $\mathbb{S}^n$ to a point from $\mathbb{S}^n\setminus X$. Then $X$ is not $\rho_{k,i}$-complete.
\end{proof}

\begin{definition} For a positive integer $n$, let us say that a set $A\subseteq\mathbb{R}^n$ is \textbf{well placed} if, for each open subset $V$ of $\mathbb{R}^n$ such that $V\cap A\neq\emptyset$, there exist $p,q\in V\cap\mathbb{Q}^n$ such that $[p, q]\cap A\neq\emptyset$ where $[p, q]$ is the standard segment in $\mathbb{R}^n$ with end-points $p$ and $q$.
\end{definition}

The following theorem leads to a partial solution to Problem 53:

\begin{theorem} For each positive integer $n>1$, the following conditions are all equivalent in $\mathbf{ZF}$:
\begin{enumerate}
\item[(i)] $\mathbf{CC}(c\mathbb{R})$;
\item[(ii)] each well placed complete subset of $\mathbb{R}^n$ is closed in $\mathbb{R}^n$;
\item[(iii)] each well placed, simultaneously connected and complete subset of $\mathbb{R}^n$ is closed in $\mathbb{R}^n$;
\item[(iv)] each well placed,  simultaneously connected and complete subspace of the plane $\mathbb{R}^2$ is closed in $\mathbb{R}^2$.
\end{enumerate}
\end{theorem}
\begin{proof} It is clear that the implications $(ii)\Rightarrow  (iii)\Rightarrow (iv)$ hold. Let us assume $(i)$ and use some ideas of the proof that (7) implies (5) in Theorem 4.55 in \cite{Her}. Consider a well placed, nonempty complete subset $X$ of $\mathbb{R}^n$. Let $J=\{ (p,q)\in\mathbb{Q}^n\times\mathbb{Q}^n: [p, q]\cap X\neq\emptyset\}$. For each pair $(p, q)\in J$, let $Y(p, q)=\{ t\in [0, 1]: (1-t)p+tq\in X\}$. Then $Y(p, q)$ is a complete subspace of $\mathbb{R}$. Under the assumption that 
$\mathbf{CC}(c\mathbb{R})$ holds, there exists $f\in\prod_{(p,q)\in J} Y(p, q)$. Since $X$ is well placed, the set $\{f(p, q): (p, q)\in J\}$ is dense in $X$. Hence $X$ is separable. Since every separable complete subset of a metric space is closed, the set $X$ is closed in $\mathbb{R}^n$. In consequence, $(i)\Rightarrow (ii)$. 

Now, suppose that $\mathbf{CC}(c\mathbb{R})$ fails. In view of Theorem 4.55 of \cite{Her}, there exists a sequentially closed subset $A$ of $\mathbb{R}$ such that $A$ is not closed in $\mathbb{R}$. Let $I=[0, 1]$ be the unit interval of $\mathbb{R}$ and let  $X=(A\times I)\cup(\mathbb{R}\times\{1\})$. Of course $X$ is well placed and connected in $\mathbb{R}^2$. Let $(z_n)_{n\in\omega \setminus \{ 0\}}$ be a Cauchy sequence in $X$ and let  $x_n, y_n$ be real numbers such that $z_n$ is the ordered pair $(x_n, y_n)$ for $n\in\omega \setminus \{ 0\}$. Put $C=\{a\in A: (\{a\}\times I)\cap\{z_n: n\in\omega \setminus \{ 0\}\}\neq\emptyset\}$. Suppose that $C$ is infinite.   Then there exists an injective subsequence $(z_{n_k})$ of $(z_n)$ such that $x_{n_k}\in A$. The sequence $(x_{n_k})$ is a Cauchy sequence of points of $A$. Let $x_0$ be the limit of $(x_{n_k})$ in $\mathbb{R}$. Since $A$ is sequentially closed, we have $x_0\in A$. Suppose that the segment $\{x_0\}\times I$ does not contain accumulation points of the sequence $(z_{n_k})$. Then, by the compactness of $\{ x_0\}\times I$,  there exists an open set $V$ in $\mathbb{R}^2$ such that $(\{x_0\}\times I)\subseteq V$ and $V\cap\{ z_{n_k}: k\in\omega \setminus \{ 0\} \}=\emptyset$. This is impossible because $x_0$ is the limit of $(x_{n_k})$. The contradiction obtained shows that the sequence $(z_{n_k})$ has an accumulation point $z_0\in\{x_0\}\times I$. Then $z_0$ is the limit of $(z_n)$ in $X$.  Now, suppose that $C$ is finite. Then the subspace $Y=(\mathbb{R}\times\{1\})\cup\bigcup_{a\in C}(\{a\}\times I)$ is closed in $\mathbb{R}^2$ and, moreover,  there exists $n_0\in\omega \setminus \{ 0\}$ such that $z_n\in Y$ for each natural number $n>n_0$. Since $Y$ is complete and closed in $\mathbb{R}^2$, the limit of $(z_n)$ in $\mathbb{R}^2$ belongs to $Y$. In consequence, $X$ is complete. Of course, $X$ is not closed in $\mathbb{R}^2$.
\end{proof}

\begin{corollary} Let $n>1$ be a positive integer. Let $i\in\{1, 2\}$ and $k\in 3$. If all simultaneously $\rho_{k,i}$- complete and connected subspaces of $\mathbb{S}^n$ are closed in $\mathbb{S}^n$, then $\mathbf{CC}(c\mathbb{R})$ holds.
\end{corollary}

Problem 53 can be regarded as equivalent to the following one:

\begin{problem} Let $n\in\omega \setminus \{ 0\}$ and $n>1$. If $\mathbf{M}$ is a model of $\mathbf{ZF}$ such that $\mathbf{CC}(c\mathbb{R})$ holds in $\mathbf{M}$, is it true in $\mathbf{M}$ that every complete subspace of $\mathbb{R}^n$ is closed in $\mathbb{R}^n$?
\end{problem}

\section{Product semimodules from SMPA}

Let $n\in\omega \setminus \{ 0\}$. For all $x,y\in\mathbb{S}^n$, one can define $x\oplus y\in\mathbb{S}^n$ and $x\otimes y\in\mathbb{S}^n$ in an obvious way:
\begin{eqnarray}
(x\oplus y)(i) &=& x(i)\oplus y(i)
\end{eqnarray}
\begin{eqnarray}
(x\otimes y)(i) &=& x(i)\otimes y(i)
\end{eqnarray}
where, for each $i\in n$, $x(i)\oplus y(i)$ and $x(i)\otimes y(i)$ are defined as in Corollary 28. Then $(\mathbb{S}^n, \oplus, \otimes)$ is a semiring. 

\subsection{$\mathbb{S}$ as a semimodule}

By a scalar we mean an element of $\mathbb{R}_{\varepsilon }$. 
For $a\in \mathbb{S}$, $\lambda \in \mathbb{R}_{\varepsilon }$, we define 
\begin{eqnarray}
\lambda \otimes a =\sgn(a) (\lambda +|a|).
\end{eqnarray} 
The operation 
$\otimes: \mathbb{R}_{\varepsilon }\times \mathbb{S\rightarrow S}$ is well defined, and we can consider
it as an outer operation from $\mathbb{R}_{\varepsilon}\times\mathbb{S}$ to $\mathbb{S}$ or the restriction to $\mathbb{R}_{\varepsilon}\times\mathbb{S}$ of the inner operation $\otimes$
in $\mathbb{S}_{\max}$ because
$\mathbb{R}_{\max }\subset \mathbb{S}_{\max}$. Considering $\otimes$ as the outer operation, we can look at $\mathbb{S}_{\max}$ as at a semimodule (``semi-vector-space'') over the semifield $\mathbb{R}_{\max}$. 
%%In this sense, we consider $\mathbb{S}$ as a semimodule.

\subsection{$\mathbb{S}^n$ as a semimodule}

Let $n\in\omega \setminus \{ 0\}$.   We consider $\mathbb{S}^n$ as the product semimodule with the inner operation $\oplus$ defined by equation (15) and the outer operation $\otimes:\mathbb{R}_{\varepsilon}\times\mathbb{S}^n\to\mathbb{S}^n$ defined as follows:

\begin{eqnarray}
(\lambda \otimes a)(i)=\lambda \otimes a(i)\text{,}
\end{eqnarray}%
where $\lambda\in\mathbb{R}_{\varepsilon},  a\in\mathbb{S}^n$ and $i\in n$, while $\lambda \otimes
a(i)$ is as in  (17).

We deduce the following two corollaries from Propositions 37 and 38, respectively:

\begin{corollary}
The operation $\otimes:\mathbb{S}^n\times\mathbb{S}^n\to\mathbb{S}^n $ is continuous in the usual topology of~$\mathbb{S}^{n}$.
\end{corollary}

\begin{corollary}
The operation $\oplus:\mathbb{S}^n\times\mathbb{S}^n\to\mathbb{S}^n $ is discontinuous in 
the usual topology of~$\mathbb{S}^{n}$.
\end{corollary}
 
\section{Convexities in $\mathbb{S}^n$}

If, for each pair of points $x,y$ of a space $X$, a segment $[x,y]$ with end-points $x$ and $y$ is defined, then a set $A\subseteq X$ is called convex if, for all $x,y\in A$, we have $[x,y]\subseteq A$ (see, for instance, \cite{H}). In particular, for $x,y\in\mathbb{S}^n$, we can consider distinct kinds of segments with end-points $x$ and $y$, among them, traditional segments, geometric segments and semimodule segments. These kinds of segments in $\mathbb{S}^n$ lead to distinct kinds of convexity in $\mathbb{S}^n$. Let us turn our attention to traditional, semimodule and geometric convexity in $\mathbb{S}^n$. One can also investigate $\rho_{k,j}$-convex sets in $\mathbb{S}^n$ for each $k\in 3$ and $j\in\{1,2\}$.

 \subsection{Traditionally convex sets in $\mathbb{S}^n$}
 
 Let $n\in\omega \setminus \{ 0\}$. A convex set in $\mathbb{C}^n$ is a set $A\subseteq\mathbb{C}^n$ such that, for each pair $x,y\in A$, the traditional segment $[x, y]$ in the affine space $\mathbb{C}^n$ is contained in $A$.
 
 \begin{definition} A set $C\subseteq \mathbb{S}^n$ will be called \textbf{traditionally convex} if  
 $[^n\Phi](C)$ is  convex in $\mathbb{C}^n$.
 \end{definition}
\begin{remark} 
Let us notice that a nonempty set $C\subseteq\mathbb{S}$ is simultaneously compact and traditionally convex in $\mathbb{S}$ if and only if $C$ is a traditional segment in $\mathbb{S}$.
\end{remark}
 All traditionally convex subsets of $\mathbb{S}^n$ are connected.

\subsection{Semimodule convex sets in $\mathbb{S}^n$}

The semimodule structure of $\mathbb{S}^n$ allows us to think about semimodule convexity of subsets of $\mathbb{S}^n$.
In analogy to conventional algebra, we define a \textbf{semimodule segment} $[a, b]_{sm}$
with end-points  $a\in\mathbb{S}^n$ and $b\in \mathbb{S}^n$ as follows:
\begin{eqnarray}
\left[ a,b\right]_{sm}=\{ \left( \lambda \otimes a\right) \oplus \left( \gamma \otimes
b\right): \lambda , \gamma \in 
\mathbb{R}_{\varepsilon }, \mbox{ with } \lambda \oplus \gamma =0 \}.
\end{eqnarray}

\begin{definition} A set\ $A\ \subseteq \mathbb{S}^n$
is said to be  \textbf{semimodule convex} 
if, for each pair $a, b$ of points of $A$, the  semimodule segment $\left[ a, b\right]_{sm}$ is contained in $A$.
\end{definition}

Convexity in arbitrary idempotent semimodules was introduced by Cohen,
Gaubert and Quadrat \cite{Co}. \ Some analogues of famous theorems of
functional analysis like Minkowski's theorem, separation theorem, and some
notions of geometry like simplices and convex polytopes were considered in
 Max-Plus Algebra (for instance, one can refer to \cite{De}, \cite{Ga}, \cite{Ka}). 

\begin{remark}
We have the following semimodule segments in $\mathbb{S}$ for $r,s\in \mathbb{R}_+$ with $r<s$:
\begin{eqnarray}
 [ \oplus r,\ominus r ]_{sm} = \{ \oplus r, \ominus r, r^{\bullet} \}   \text{,} 
 \end{eqnarray}
 \begin{eqnarray}
 [ \oplus r, r^{\bullet} ]_{sm} = \{  \oplus r, r^{\bullet} \}  \text{,}  
 \end{eqnarray}
 \begin{eqnarray}
  [\ominus  r,\oplus s ]_{sm} =\{\ominus r, r^{\bullet} \} \cup \{\oplus t: r<t\le s  \} 
\end{eqnarray}
 \begin{eqnarray}
  [  r^{\bullet} ,\oplus s ]_{sm} =\{ r^{\bullet} \} \cup \{\oplus t: r<t\le s  \} 
\end{eqnarray}
 \begin{eqnarray}
  [\ominus  r, s^{\bullet} ]_{sm} =\{\ominus r \} \cup \{ t^{\bullet}: r\le t\le s  \} 
\end{eqnarray}
Other cases of semimodule segments in $\mathbb{S}$ are similar or obvious.

Notice that the intersection of two semimodule segments may not be a semimodule segment. Also, if $[a, b]_{sm}$ and $[c, d]_{sm}$ are semimodule segments in $\mathbb{S}$ such that $[a, b]_{sm}\cap [c, d]_{sm}\neq\emptyset$, then, in general, the union $[a, b]_{sm}\cup [c, d]_{sm}$ need not be a semimodule segment.  However, if $c\in [a,b]_{sm}$, then $[a,b]_{sm}=[a,c]_{sm}\cup [c,b]_{sm}$. All disconnected semimodule segments in $\mathbb{S}$ have either two or three connected components. 
\end{remark}

\begin{example}
Consider the semimodule segment in $\mathbb{S}$ between $a=\oplus 1$ and $b=\ominus 0$.
Then $[a,b]_{sm}=\{b\} \cup \{0^{\bullet}\} \cup \{ \oplus t: t\in (0,1]\}$, so the segment $[a, b]_{sm}$ has three connected components. Moreover, this segment is not closed in $\mathbb{S}$. For $i\in\{1, 2\}$,  both $\ominus 0$ and $0^{\bullet}$ are $d_i$-nearest points to $\varepsilon^{\bullet}$ in $[a,b]_{sm}$. More precisely, $P_{d_i, [a, b]_{sm}}(\varepsilon^{\bullet})=\{\ominus 0, 0^{\bullet}\}$ for $i\in\{1, 2\}$.
\end{example}

\begin{example}
A semimodule segment in $\mathbb{S}^2$ can have five connected components. For example, consider
\begin{eqnarray*}
C=[(\oplus 0, \ominus 1),(\ominus 1, \oplus 0)]_{sm}=\{ (\oplus 0,\ominus 1),(\ominus 1, \oplus 0),(\ominus 1, 0^{\bullet}),(0^{\bullet}, \ominus 1)\} \cup \\
\cup \{(\ominus t, \ominus s): \max(t,s)=1, t,s  >0. \}
\end{eqnarray*}
Then $C$ has exactly four isolated points and exactly one connected component of $C$ is of the form $[^2\Phi]^{-1}(L)$ where $L\subseteq\mathbb{C}$ is a broken line with its end-points deleted.
\end{example}

\subsection{Geometric convexity in $\mathbb{S}^n$}

\begin{definition} A set\ $A\ \subseteq \mathbb{S}^n$
is said to be  \textbf{geometrically convex} 
if, for each pair $a, b$ of points of $A$, the  geometric segment $\left[ a, b\right]_{g}$ is contained in $A$.
\end{definition}

The notions of a geometrically convex set and a $D_2$-convex set in $\mathbb{S}^n$ coincide. Of course, each traditionally convex subset of $\mathbb{S}^n$ is geometrically convex but not every geometrically convex set in $\mathbb{S}^n$ is traditionally convex. Since geometric segments in $\mathbb{S}^n$ are connected sets, each geometrically convex subset of $\mathbb{S}^n$ is connected.  In the case of $\mathbb{S}$, we can state the following propositions:

\begin{proposition} The geometric segment $[a, b]_g$ between  points $a,b\in \mathbb{S}$ is the intersection of all connected subsets of $\mathbb{S}$ that contain both $a$ and $b$. 
\end{proposition}

\begin{proposition}
 A set $A\subseteq\mathbb{S}$ is geometrically convex if and only if $A$ is connected.
\end{proposition}

The proofs to Propositions 68 and 69 are so easy that we omit them. Of course, if $n>1$ then, for distinct $a, b\in\mathbb{S}^n$, the intersection of all connected subsets of $\mathbb{S}^n$ that contain both $a$ and $b$ is the two-point set $\{a, b\}$. 

Let us recall that, if $X$ is a topological space, $a, b\in X$, while $C$ is a continuum in $X$ such that $a,b\in C$, then $C$ is called irreducible between $a$ and $b$ if, for each continuum $D\subseteq C$ such that $a,b\in D$, the equality $D=C$ holds. The axiom of choice is used in the standard proof of the theorem of $\mathbf{ZFC}$ asserting that, for each pair of distinct points of a Hausdorff continuum $X$ there exists a continuum $C$ in $X$ such that $C$ is irreducible between $a$ and $b$. 

\begin{problem} Is there a model $\mathbf{M}$ of $\mathbf{ZF}$ such that there exists a Hausdorff continuum $X$ in $\mathbf{M}$ which has the property that, for a pair $a, b$ of distinct points of $X$, there does not exist in $\mathbf{M}$ a subcontinuum $C$ of  $X$ such that $C$ is irreducible between $a$ and $b$?
\end{problem}

\begin{remark} Suppose that we work inside a model $\mathbf{M}$ of $\mathbf{ZFC}$.  Let $X$ be a semidendrite in $\mathbf{M}$ and let $a,b\in X$. We can generalize properties of dendrites given on page 125 in \cite{M} to semidendrites. Namely, we can prove that there exists in $\mathbf{M}$ a unique continuum $C\subseteq X$ such that $a, b\in C$ and $C$ is irreducible between $a$ and $b$. Let us denote this continuum $C$ by $[a, b]_{c-g}$ and call it the \textbf{c-geometric segment} of $X$ between $a$ and $b$.  Now, assume, in addition, that $X$ is metrizable in $\mathbf{M}$ and that $a\neq b$. Then  $[a, b]_{c-g}$ is an arc with end-points $a$ and $b$ (see Lecture VIII of \cite{M}).  Notice that $[a, b]_{c-g}$ is the intersection of all connected subspaces of $X$ that contain both $a$ and $b$, so $[a, b]_{c-g}$ is the smallest (with respect to inclusion) among all connected subsets of $X$ that contain both $a$ and $b$. Let us call a set $A\subseteq X$ \textbf{c-geometrically convex} in $X$ if, for each pair $a, b$ of distinct points of $A$, the inclusion $[a, b]_{c-g}\subseteq A$ holds. Then a set $A\subseteq X$ is connected if and only if it is c-geometrically convex in $X$. If $X=\mathbb{S}$, then $\mathbf{M}$ can be any model of $\mathbf{ZF}$.
\end{remark} 

\begin{remark} For a pair $a,b$ of points of $\mathbb{R}^2$, let $\mathcal{C}(a,b)$ be the collection of all connected subsets of $\mathbb{R}^2$ which contain both $a$ and $b$. Then each continuum $C\in\mathcal{C}(a,b)$ such that $C$ is irreducible between $a$ and $b$ is a minimal (with respect to inclusion) element of $\mathcal{C}(a, b)$. Since $\bigcap\{ C: C\in\mathcal{C}(a,b)\}=\{a, b\}$,  to define a c-geometric segment in 
$\mathbb{R}^2$ between $a$ and $b$ as the intersection $\bigcap\{ C: C\in\mathcal{C}(a,b)\}$  does not lead to anything reasonable. However, it may be still interesting to investigate in deep all topological spaces $X$ which have the property that, for each pair $a, b$ of points of $X$, the intersection of all connected subsets of $X$ that contain both $a$ and $b$ is connected in $X$. 
\end{remark}

\begin{remark}
Let $[a,b]_g$ and $[c, d]_g$ be geometric segments in $\mathbb{S}$. Then the intersection $[a, b]_g\cap [c, d]_g$ is a geometric segment in $\mathbb{S}$. Also, if $c\in [a,b]_g$, then $[a,b]_g=[a,c]_g\cup [c,b]_g$. However, if $[a, b]_g\cap [c,d]_g\neq\emptyset$, then $[a, b]_g\cup [c, b]_g$ is a dendrite which need not be a geometric segment in $\mathbb{S}$. 
\end{remark}

\section{The minimizing vector theorem in $\mathbb{S}^n$}

In what follows, all Hilbert spaces are assumed to be over the field $\mathbb{R}$.  Given a Hilbert space $H$, we consider $H$ as a metric space equipped with the metric $\rho_H$ of $H$ which is induced by the inner product of $H$, i. e. $\rho_H(x, y)= \left\Vert x-y\right\Vert $ where $\left\Vert \cdot \right\Vert$ is the norm defined by the inner product of $H$. The minimizing vector theorem, also called the Hilbert projection theorem, is a famous result in
convex analysis saying that it holds true in $\mathbf{ZF+CC}$ that, for every point $x\in H$ in a Hilbert space  $H$
and every nonempty closed convex subset $C$ of  $H$, there exists a unique point $y\in C$
for which $\left\Vert x-y\right\Vert $ is the distance of $x$ from $C$.  In particular, this is true in every model $\mathbf{M}$ of $\mathbf{ZF+CC}$ for any closed affine subspace $C$ of $H$ provided $H$ is in $\mathbf{M}$ and we work on $H$ inside $\mathbf{M}$. In the case when $C$ is a closed affine subspace of the $H$, a necessary and sufficient condition for $y\in C$ to be the nearest point to $x$ in $C$  is that the vector $x-y$ be orthogonal to $C$.  A result of $\mathbf{ZF}$ is the following version of the vector minimizing theorem:

\begin{theorem} For every Hilbert space $H$, for each nonempty, simultaneously convex and Cantor complete subset $C$ of $H$ and for each $x\in H$, there exists a unique $y\in C$ such that $\left\Vert x-y\right\Vert $ is the distance of $x$ from $C$.
\end{theorem}

As we have already mentioned in subsection 6.1, according to Theorem 7 of \cite{Ker}, in a model of $\mathbf{ZF}$,  a complete subset of a metric space need not be Cantor complete. Therefore, one may ask whether it holds true  in every model $\mathbf{M}$ of $\mathbf{ZF}$ that all simultaneously complete and convex subsets of a Hilbert space $H$ in $\mathbf{M}$ are closed in $H$. A partial answer to this question is given in the following proposition:

\begin{proposition} Let $H$ be a finitely dimensional Hilbert space and let $C$ be a nonempty convex complete subset of $H$. Then $C$ is closed in $H$.
\end{proposition}

\begin{proof} Let $Y$ be the intersection of all affine subspaces of $H$ that contain $C$. Since $H$ is finitely dimensional, it is known that $\text{int}_Y(C)\neq\emptyset$. Let $a\in\text{int}_Y(C)$ and $b\in\text{cl}_Y(C)$. Then the interval $I=[a, b]=\{(1-t)a+tb: t\in\mathbb{R},  0\leq t\leq 1\}$ is contained in $\text{cl}_Y(C)$, while $I\setminus\{ b\}$ is a subset of $\text{int}_{Y}(C)$. Of course, there exists a sequence $(x_n)$ of points of $I\setminus\{b\}$ which converges in $H$ to $b$. Since $(x_n)$ is a Cauchy sequence of points of $C$, we have $b\in C$ because $C$ is complete. Hence $C=\text{cl}_Y(C)=\text{cl}_H(C)$. 
\end{proof}

At this moment, we do not know a satisfactory solution to the following problem:

\begin{problem} Is is true in $\mathbf{ZF}$ that, for an arbitrary infinitely dimensional Hilbert space $H$, all simultaneously convex and complete subsets of $H$ are Cantor complete in $H$?
\end{problem}

We recommend \cite{FM} as one of the newest survey articles on the minimizing vector theorem. We are going to state some simple versions of this theorem for $\mathbb{S}^n$.

\subsection{A characterization of Chebyshev sets in $\mathbb{S}$}

For $i\in\{1, 2\}$, the following theorem gives a complete characterization of $d_i$-Chebyshev sets in $\mathbb{S}$:

\begin{theorem}
For each nonempty  closed subset $C$ of $\:\mathbb{S}$,
 the following conditions are equivalent:
\begin{enumerate}
\item[(i)] $C$ is connected;
\item[(ii)] $C$ is a $d_1$-Chebyshev set;
\item[(iii)] $C$ is a $d_2$-Chebyshev set.
\end{enumerate}
\end{theorem}
\begin{proof}  Assume $C$ is connected and $x\in\mathbb{S}$.  Consider the following two cases (a) and (b):

(a)  $\varepsilon^{\bullet}\notin C$. Then $C$ is contained in one of the lines $\mathbb{S}^{\oplus}$, $\mathbb{S}^{\ominus}$, $\mathbb{S}^{\bullet}$, where connnectedness is equivalent with traditional convexity, so $P_{d_1, C}(x)=P_{d_2, C}(x)$ and $P_{d_1, C}(x)$ is a singleton.

(b) $\varepsilon^{\bullet} \in C$. Then $C$ is star-like; hence,  it is easily seen that $P_{d_1, C}(x)=P_{d_2, C}(x)$ and $P_{d_1, C}(x)$ is a singleton which lies on the same line as $x$.  

In consequence,  $(i)$ implies both $(ii)$ and $(iii)$. On the other hand, if $C$ were disconnected, there would exist a connected component $K$ of $\mathbb{S}\setminus C$, distinct points $a_1, a_2\in C\cap\text{cl}_{\mathbb{S}}(K)$  and points $y_1, y_2\in K$ such that $a_1, a_2\in P_{d_i, C}(y_i)$ for $i\in\{1, 2\}$. Therefore, each of $(ii)$ and $(iii)$ implies $(i)$. 
\end{proof}

\begin{corollary} If $C$ is a nonempty, simultaneously connected and closed subset of $\mathbb{S}$, then $p_{d_1, C}=p_{d_2, C}$.
\end{corollary}

\begin{proposition}
For each nonempty closed set $C\subseteq \mathbb{S}$, that $C$ is geometrically convex in $\mathbb{S}$ is equivalent to each one of conditions $(i)-(iii)$ of Theorem 76.
\end{proposition}

Let us have a brief look at semimodule convex closed subsets of $\mathbb{S}$. 

\begin{example}
The semimodule segment  $A=[\oplus 0,\ominus 0]_{sm}=\{ \oplus 0, \ominus 0, 0^{\bullet}\}$ is a closed, semimodule convex set in $\mathbb{S}$ with three connected components.  For each $i\in\{1, 2\}$, we have $P_{d_i, A}(\varepsilon^{\bullet})=A$.
\end{example}

\begin{proposition}
Let $C$ be a nonempty, simultaneously semimodule convex and closed subset of $\mathbb{S}$. Then $C$ is $d_i$-proximinal and, for each $x\in\mathbb{S}$, the set $P_{d_i, C}(x)$ consists of at most three points. 
\end{proposition}

\begin{proof}
This is clear since, on each of the lines $\mathbb{S}^{\oplus}$, $\mathbb{S}^{\ominus}$, $\mathbb{S}^{\bullet}$, semimodule convexity agrees with the traditional convexity.
\end{proof}

\subsection{Chebyshev sets in $\mathbb{S}^{n}$}

\begin{definition}
A set $A\subseteq \mathbb{S}^{n}$ is said to be \textbf{box semimodule convex} if it
is a Cartesian product $\prod_{i\in n} A_i$ of semimodule convex subsets  $A_i$ of $\mathbb{S}$. 
\end{definition}

We can generalize Proposition 81 as follows:

\begin{proposition}
Let $j, k\in\{1, 2\}$. Suppose that $A$ is a nonempty, closed,  box semimodule convex subset of $\mathbb{S}^{n}$. Then $A$ is $\rho_{k,j}$-proximinal and, moreover, for each $x\in\mathbb{S}^n$, the set $P_{\rho_{k, j}, A}(x)$ consists of at most $3^{n}$ points.
\end{proposition}
\begin{proof} Since closed balls in $(\mathbb{S}^n, \rho_{k, j})$ are compact, each nonempty closed set in $\mathbb{S}^n$ is $\rho_{k,j}$-proximinal. That $P_{\rho_{k, j}, A}(x)$ consists of at most $3^{n}$ points follows directly from Theorem 19 (i) and from Proposition 81.
\end{proof}

\begin{example}
In view of Example 80,  the set $A=\{ \oplus 0, \ominus 0, 0^{\bullet}  \}^n$ in $\mathbb{S}^n$ is box semimodule convex. Let $x_0\in\mathbb{S}^n$ be defined by: $x_0(i)=\varepsilon^{\bullet}$ for each $i\in n$. Then, if $k\in 3$ and $j\in\{1, 2\}$, the set $P_{\rho_{k, j}, A}(x_0)$ consists of exactly $3^n$ points. 
\end{example}

\begin{theorem}
Let $j,k \in\{1, 2\}$. Suppose that $A=\prod_{i\in n}A_i$ where each $A_i$ is nonempty and closed in $\mathbb{S}$. Then $A$ is $\rho_{k, j}$-Chebyshev in $\mathbb{S}^n$ if and only if  $A$ is connected. 
\end{theorem}
\begin{proof} It suffices to apply Theorems 19(ii) and 77.
\end{proof}

The following problem is still unsolved:
\begin{problem} Let $i\in\{1, 2\}$ and $A$ be a nonempty closed subset of $\mathbb{S}^n$ where $n>1$.  Is it true that $A$ is $D_i$-Chebyshev if and only if $A$ is geometrically convex in $\mathbb{S}^n$?
\end{problem}

Let us leave Problem 86 for future investigations in another work.

\textbf{Acknowledgement.}
This research was supported by King Abdulaziz University. The article is a part of research project initiated by Hanifa Zekraoui.

\end{document}